\documentclass[12pt]{amsart}

\def\@typesizes{
       \or{5}{6.5}\or{6}{7.5}\or{7}{8.5}\or{8}{11}\or{9}{12}
       \or{10}{13}
       \or{\@xipt}{14}\or{\@xiipt}{15}\or{\@xivpt}{18}
       \or{\@xviipt}{20}\or{\@xxpt}{24}}

\usepackage[utf8]{inputenc}
\usepackage[T1]{fontenc}

\usepackage{hyperref}
\usepackage{graphicx}
\usepackage{amsmath,amsfonts,amssymb,amsthm,cite}

\numberwithin{equation}{section}
\newtheorem{theorem}{Theorem}[section]
\newtheorem{corollary}[theorem]{Corollary}

\newtheorem{example}[theorem]{Example}
\newtheorem{definition}[theorem]{Definition}
\newcommand{\nonprint}[1]{}

\begin{document}

\title[Parameter-dependent boundary-value problems]{Parameter-dependent inhomogeneous boundary-value problems in Sobolev spaces}

\author{Olena Atlasiuk}
\address{
University of Helsinki, Department of Mathematics and Statistics, P.O. Box 68, Pietari Kalmin katu 5, 00014 Helsinki, Finland and \\
Institute of Mathematics of the National Academy of Sciences of Ukraine, st. Tereschenkivska 3, 01024 Kyiv, Ukraine }

\email{olena.atlasiuk@helsinki.fi}

\author{Vladimir Mikhailets}
 \address{King's College London, Strand, WC2R 2LS London, UK and Institute of Mathematics of the National Academy of Sciences of Ukraine, st. Tereschenkivska 3, 01024 Kyiv, Ukraine }

\email{vladimir.mikhailets@gmail.com}

\author{Jari Taskinen}
\address{University of Helsinki, Department of Mathematics and Statistics, P.O. Box 68, Pietari Kalmin katu 5, 00014 Helsinki, Finland}

\email{jari.taskinen@helsinki.fi}

\begin{abstract}
We study a wide class of linear inhomogeneous boundary-value problems for $r$th order ODE-systems depending on a
parameter $\mu$ belonging to a general metric space $\mathcal M$. The solutions belong to
the Sobolev spaces $(W^{n+r}_p)^m$, $n\in\mathbb{N}\cup\{0\}$, $m, r \in \mathbb{N}$, $1\leq p\leq \infty$.
The boundary conditions are of a most general form $By=c$, where $B$ is an arbitrary continuous operator
from $(W^{n+r}_p)^m$ to $\mathbb{C}^{rm}$. Thus, they  may contain derivatives of the unknown vector
function of integer and/or fractional orders $\geq r$. We find necessary and
sufficient conditions for the continuity of solutions with respect to the parameter $\mu$. We also prove
that the solutions of the original problems can be approximated in the space $(W^{n+r}_p)^m$, $p<\infty$, by
solutions of ODE-systems $Ly=f$ with polynomial coefficients, polynomial right-hand sides, and multipoint boundary
conditions, where the approximating operators $L_k$ and $B_k$ can be chosen independently of the right-hand sides.
\end{abstract}

\maketitle



\maketitle

\tableofcontents

\section{Introduction}\label{Sec.1}
One of the central contemporary questions in the theory of systems of ordinary differential equations
concerns parameter-dependent systems, in particular,
the continuous dependence on the parameter or the limit behavior of solutions. Regarding classical fundamental results on continuous dependence on a parameter, we mention the results of Gikhman \cite{Gikhman}, Krasnoselskii and Krein \cite{KrasnKrein}, Kurzweil and Vorel \cite{KurzweilVorel} concerning the Cauchy problem for nonlinear differential systems. For linear systems, these results were improved and generalized by Levin \cite{Levin}, Opial \cite{Opial}, Reid \cite{Reid}, Nguen \cite{Nguen}, and more recently, by the second named author and his collaborators
\cite{MikhMurachSol2016, AtlasiukMikhailets20191, Boyko, AtlasiukMikhailets20192, NovAM2023, AtlasiukMikhailets2024, KodliukMikhailets2013}.

Unlike solutions to the Cauchy problem, solutions to boundary-value problems may not exist or may not be
unique. Questions about the solvability of linear boundary-value problems and,  in particular, their
Fredholm  indices and  $d$-characteris\-tics in various function spaces were studied in the papers \cite{NovAM2023, Pelekhata18}.
The parameter dependence of boundary-value problems has attracted  much less attention than the corresponding
question for the Cauchy problem. For a long time, limit theorems for parameter dependence  were
established only for inhomogeneous boundary conditions in boundary-value problems
\begin{equation}
Ly=f, \quad By=c,   \label{1.2}
\end{equation}
where $L$ is a system of $m$ linear differential equations of order $r \in \mathbb{N}
= \{1,2,3, \ldots\}$ with summable coefficients, and $B$ is a linear continuous finite-dimensional operator
$$B \colon (W^r_1)^m \rightarrow \mathbb{C}^{rm}.$$
In this case, the general boundary conditions may only contain derivatives of order $\leq r-1$
of the unknown function  \cite{Krall,Iha,Chua,Goodrich,Whyburn,Smogorzhevsky,Krein82,Skubachevskii,Ma,Adomian,Krall68, Providas}.
Sufficient conditions for the continuity of solutions in the case $r=1$ were proved in the articles of Kiguradze \cite{Puza, Kigyradze2003,Kigyradze1975,Kigyradze1987} and Ashordia \cite{Ashordia97, Ashordia,Ashordia20}. These results were improved and generalized to systems of differential equations of order $r\geq2$ in the series of works \cite{Pelekhata19,KodliukMikhailetsReva2013}.

In certain applications, particularly in optimization theory, there naturally arise problems with boundary conditions that contain higher derivatives than the order of the differential equation. If the coefficients of the system of order $r \in \mathbb{N}$ belong to $(W_p^n)^{m \times m}$ and its right-hand side belongs to $(W_p^n)^m$
($1\leq p\leq \infty$, $n\in \mathbb{N}$) instead of just being in $(L_p)^{m\times m}$ and $(L_p)^m$, respectively, then
all solutions of the differential system have additional smoothness and belong to the Sobolev space $(W_p^{n+r})^m$. For this case, we introduce new \textit{generic} inhomogeneous boundary conditions in Sobolev spaces
in the form of the operator equation $By=c$, where
\begin{equation*}
B\colon (W^{n+r}_p)^m\rightarrow\mathbb{C}^{\ell}
\end{equation*}
is a continuous operator and $\ell$ is the number of linearly independent scalar boundary conditions.
The resulting conditions can be either overdetermined or underdetermined.
This new direction of research
has been developed in a series of works by the second named author and his co-authors
\cite{AtlasiukM2026, HRYPMikhailetsMurach2017,HRYPKodliukMikhailets2015,HRYP}. The applications include multipoint
boundary-value problems \cite{Ashordia98, Ashordia05, Atlasiuk2020, Atlasiuk20201}, as well as problems of the spectral theory of differential operators with singular
coefficients \cite{Horyunov}. Note that for such boundary-value problems, the formal conjugate problem and
Lagrange formula are not defined; therefore, novel approaches and methods are required.
Generic boundary-value problems have also been studied in other spaces of differentiable functions
\cite{Masliuk18, Masliuk17, Sold2015, Masliuk,Murach21}. It is worth
noting that even in the case of $n=0$, $p=1$, the class of generic boundary-value problems is wider than the
class of problems \eqref{1.2}.

The article is organized as follows.
Section~\ref{Sec.2} contains the formulation and proofs of  general results on the continuous
parameter dependence, and in Section~\ref{Sec.3} we prove limit theorems for the operators appearing
in the boundary-value problems. Moreover, in Theorem \ref{opBinSobolev}
we find a concrete representation for the abstract boundary
operator of the generic boundary-value problem, which leads to a more detailed formulation of the
results of Section~\ref{Sec.2}. Section~\ref{Sec.4} contains applications of the main results, namely,
Theorem  \ref{thapro} on the approximation of solutions of an arbitrary boundary-value problem by solutions
of a system  with polynomial coefficients, right-hand sides of the equation, and multipoint boundary conditions.

\section{General results on the continuous parameter dependence}\label{Sec.2}

Let $(a,b)\subset\mathbb{R}$ be a finite interval and let the parameters
$$
n\in\mathbb{N}\cup\{0\},\quad\{m, r, k\}\subset\mathbb{N},\quad  1\leq p\leq \infty,
$$
be arbitrary. We denote by
\begin{align*}
&W_p^{n+r}\bigl([a,b];\mathbb{C})\\
&:= \bigl\{y\in C^{n+r-1}([a,b];\mathbb{C})\colon y^{(n+r-1)}\in AC[a,b], \, y^{(n+r)}\in L_p[a,b]\bigr\}
\end{align*}
the usual  complex Sobolev space and we set $W_p^0:=L_p$. This space is a Banach space with respect to the norm
$$
\bigl\|y\bigr\|_{n+r,p}=\sum_{s=0}^{n+r}\bigl\|y^{(s)}\bigr\|_{p},
$$
where $\|\cdot\|_p$ stands for the norm in the Lebesgue space $L_p\bigl([a,b]; \mathbb{C}\bigr)$. We
will need the Sobolev spaces
$$
(W_p^{n+r})^{m}:=W_p^{n+r}\bigl([a,b];\mathbb{C}^{m}\bigr)
\;\;\mbox{and}\;\;
(W_p^{n+r})^{m\times m}:=W_p^{n+r}\bigl([a,b];\mathbb{C}^{m\times m}\bigr),
$$
which consist, respectively, of vector- and matrix-valued functions with ele\-ments belonging to $W_p^{n+r}$.
The norms in these spaces are defined to be the sums of the Sobolev norms of the components, and the
same notation $\|\cdot\|_{n+r,p}$ is used in all cases, which will be clear from the context.
The same convention will be applied to all other Banach spaces. Recall that, if $p<\infty$, the
Sobolev spaces in question are separable and have a Schauder basis.

Let $\mathcal{M}$ be an arbitrary metric space. In the sequel, $\mu \in \mathcal{M}$ will denote a free
parameter whereas $\mu_0 \in \mathcal{M}$ denotes an arbitrarily  fixed one.
We consider the following linear boundary-value problem for an unknown
vector-valued function $y(\cdot,\mu)\in (W_{p}^{n+ r})^m$,
\begin{equation}\label{bound_z1}
\left(L(\mu)y(\mu)\right)(t):=y^{(r)}(t,\mu) + \sum\limits_{\ell=1}^rA_{r-\ell}(t,\mu)y^{(r-\ell)}(t,\mu)=f(t,\mu), \quad t\in(a,b),
\end{equation}
\begin{equation} \label{bound_z2}
B(\mu)y(\mu)=c(\mu),
\end{equation}
where matrix-valued functions $A_{r-\ell}(\cdot,\mu) \in (W_p^n)^{m\times m}$, a vector-valued function
$f(\cdot,\mu)$ $\in (W^n_p) ^m$, a vector $c(\mu) \in \mathbb{C}^{rm}$ and a
continuous linear operator
\begin{equation}\label{oper_B(e)v}
B(\mu)\colon (W^{n+r}_p)^m\rightarrow\mathbb{C}^{rm}
\end{equation}
are given and arbitrary.
The boundary condition \eqref{bound_z2} consists of $rm$ scalar conditions for a system of $m$ differential
equations of $r$-th order. In the sequel, we will use the rules of standard matrix algebra; thus,
vectors and vector-valued functions will be presented as columns.

A solution to problem \eqref{bound_z1}, \eqref{bound_z2} is understood as a vector-valued function
$y \in (W_{p}^{n+r})^m$ which satisfies both equation \eqref{bound_z1} (everywhere if $n\geq 1$, and
almost everywhere if $n=0$) on $(a,b)$ and equality \eqref{bound_z2}.
As explained above, we  call the general boundary condition \eqref{bound_z2} with an arbitrary continuous
operator \eqref{oper_B(e)v} \textit{generic} for the differential system \eqref{bound_z1}. It covers
all classical types of boundary conditions, such as initial conditions in the Cauchy problem, various
multipoint conditions, integral conditions,  mixed boundary conditions, as well as non-classical conditions
containing fractional derivatives, where the order of the derivatives may exceed the order of the
differential equation. Finally, we do not make any a priori assumption on the regularity of
the matrix-valued functions $A_{r-\ell}(t,\mu)$ with respect to $\mu$.

We write problem \eqref{bound_z1}, \eqref{bound_z2} in the form of a linear operator equation
\[ \big(L(\mu),B(\mu)\big)y(\mu)=\big(f(\mu),c(\mu)\big), \]
where $\big(L(\mu),B(\mu)\big)$ is the family of continuous linear operators
\begin{equation}\label{(L,B)vp}
\big(L(\mu),B(\mu)\big) \colon (W^{n+r}_p)^m\to (W^{n}_p)^m\times\mathbb{C} ^{rm}.
\end{equation}

Let $E_{1}$ and $E_{2}$ be Banach spaces. A linear bounded operator $T\colon E_{1}\rightarrow E_{2}$ is called a Fredholm operator if its kernel and co-kernel are finite-dimen\-si\-onal. If $T$ is a Fredholm operator, then its range $T(E_{1})$ is closed in $E_{2}$, and its index
$$
\mathrm{ind}\,T:=\dim\ker T-\dim\big(E_{2}/T(E_{1})\big)\in \mathbb{Z} = \{0, \pm 1, \pm2 , \ldots \}
$$
is finite (see, e.g., \cite[Lemma~19.1.1]{Hermander1985}).
According to \cite[Theorem 1]{NovAM2023}, all operators in the family  \eqref{(L,B)vp} are Fredholm
with index zero for every  $\mu$.

\begin{definition}\label{defin_vp}
We say that solutions to the boundary-value problem \eqref{bound_z1}, \eqref{bound_z2} depend continuously
on the parameter $\mu$ at a limit point $\mu_0$ of the metric space $\mathcal{M}$,
if the following two conditions are satisfied:
\begin{itemize}
\item[$(\ast)$] There exists a positive number $\varepsilon$ such that, for all $\mu\in \mathcal{B}(\mu_0, \varepsilon)$ and arbitrary right-hand sides $f(\cdot;\mu)\in (W^{n}_p)^{m}$ and $c(\mu)\in\mathbb{C}^{rm}$, the problem has a unique solution $y(\cdot;\mu)$ in the space  $(W^{n+r}_p)^{m}$;
\item [$(\ast\ast)$] The convergence of the right-hand sides $f(\cdot;\mu)\to f(\cdot;\mu_0)$ in $(W_p^n)^{m}$ and $c(\mu)\to c(\mu_0)$ in $\mathbb{C}^{rm}$ as $\mu\to\mu_0$ implies the convergence of the solutions
\begin{equation*}\label{4.guv}
y(\cdot,\mu)\to y(\cdot,\mu_0)\quad\mbox{in}\quad (W^{n+r}_p)^{m} \quad\mbox{as}\quad\mu\to\mu_0
\end{equation*}
for arbitrary $f(\cdot;\mu_0)\in (W^{n}_p)^{m}$ and $c(\mu_0)\in\mathbb{C}^{rm}$.
\end{itemize}
\end{definition}

Throughout this article, we will assume that the following condition (0) for the point
$\mu_0 \in \mathcal{M}$ is fulfilled.

\medskip

\noindent \textbf{Condition (0)}. {\it The homogeneous boundary-value problem
\begin{equation*}
L(\mu_0)y(t,\mu_0)=0,\  t\in (a,b), \ \quad B(\mu_0)y(\cdot,\mu_0)=0
\end{equation*}
has only the trivial solution $y(t,\mu_0)=0$, $t\in (a,b)$. }

\medskip

\noindent We will also consider the following two conditions on the left-hand sides of the problem
\eqref{bound_z1}, \eqref{bound_z2}.

\medskip

\noindent \textbf{Limit Conditions as $\mu\to\mu_0$}:

\smallskip

\noindent
(I) {\it $A_{r-\ell}(\cdot;\mu)\to A_{r-\ell}(\cdot;\mu_0)$ in the space $(W^{n}_p)^{m\times m}$ for every  $\ell\in\{1,\ldots, r\}$;}

\smallskip

\noindent
(II) {\it $B(\mu)y\to B(\mu_0)y$ in the space $\mathbb{C}^{rm}$ for all $y\in(W^{n+r}_p)^m$.
}

\medskip
Now, we can formulate necessary and sufficient conditions for the continuity of the solutions to
the boundary-value problem \eqref{bound_z1}, \eqref{bound_z2} with respect to an abstract parameter.

\begin{theorem}\label{nep v}
The solution to the boundary-value problem \eqref{bound_z1}, \eqref{bound_z2} depends continuously on the parameter~$\mu$ at $\mu_0\in \mathcal{M}$ if and only if this problem satisfies Condition \textup{(0)} and Limit Conditions~\textup{(I)} and~\textup{(II)}.
\end{theorem}

\begin{corollary}
If Condition \textup{(0)} and Limit Conditions~\textup{(I)} and~\textup{(II)} are satisfied for all
$\mu \in \mathcal{M}$  and the right-hand sides $f$ and  $c$ are fixed,
then the solution to the boundary-value problem \eqref{bound_z1}, \eqref{bound_z2} exists and
is unique for every $\mu \in \mathcal{M}$ and  belongs to the space $C\big(\mathcal{M}; (W^{n+r}_p)^{m}\big)$.
\end{corollary}

It is worth noting that using an arbitrary metric space $\mathcal{M}$ in Theorem \ref{nep v} yields
a unified approach to both continuous and discrete parameters.

In the case of $r=1$, $\mathcal{M} = [0, \varepsilon_0]$, $\varepsilon_0>0$, $\mu_0=0$, Theorem \ref{nep v} was proved in  \cite[Theorem 1]{AtlasiukMikhailets20192}.

To prove Theorem \ref{nep v}, we will need the homeomorphism theorem, which is proved in article \cite[Theorem 3]{AtlasiukMikhailets20191}. For the convenience of readers, we will give the formulation of this result for our case.

Let us introduce a metric space of non-degenerate matrix-valued functions
$$ \mathcal{Y}_{p}^{n+1}:=\{Y(\cdot)\in (W_{p}^{n+1})^{m\times m}\colon \,Y(a)=I_{m},\quad
  \det Y(t)\neq 0, \quad t\in [a, b]\}$$
\noindent with the metric
 $$d_{n+1, p}(Y,Z):=\|Y(\cdot)-Z(\cdot)\|_{n+1,p}.$$

\begin{theorem}\label{th1}
A nonlinear mapping
\begin{equation*}\label{rr6}
A(\cdot)\mapsto Y(\cdot),
\end{equation*}
which associates each matrix-valued function $A(\cdot) \in (W_p^{n})^{m\times m}$ with a unique solution $Y(\cdot)$ of the Cauchy matrix problem
\begin{equation*}\label{r3}
Y'(t)+A(t)Y(t)=0,\quad t\in (a,b), \quad Y(a)=I_{m},
\end{equation*}
is a homeomorphism of the Banach space $(W_{p}^{n})^{m\times m}$ onto the metric space $\mathcal{Y}_{p}^{n+1}$.
\end{theorem}

\begin{proof}[Proof of Theorem \ref{nep v}]

Let us prove \emph{the necessity} in Theorem \ref{nep v}. So, we assume that  the solution of the
boundary-value problem \eqref{bound_z1}, \eqref{bound_z2} depends continuously on $\mu$
in the sense of  Definition \ref{defin_vp}.
Then, Condition~(0) obviously holds true, and it remains to prove that the problem also satisfies Limit
Conditions (I) and~(II). We divide the  proof into three steps.

\emph{Step 1.} We consider  Condition \textup{(I)}. To show that it holds, we use the canonical reduction of
the  system to a  system of first order differential equations (see, e.g., \cite{Cartan1971}). Denoting
by $\top$ the matrix transposition, we put
\begin{gather}\label{yx}
x(\cdot,\mu):= \bigl(y(\cdot,\mu),y'(\cdot, \mu),\ldots,y^{(r-1)}(\cdot, \mu)\bigr)^\top\in(W^{n+1}_p)^{rm},\\
g(\cdot,\mu):= \bigl(\underbrace{0, \dots, 0}_{(r-1)m},f(\cdot, \mu)\bigr)^\top
\in(W^{n}_p)^{rm},\notag\\
c(\mu):= \bigl(c_{1}(\mu),\ldots,c_{r}(\mu)\bigr)^\top
\in\mathbb{C}^{rm}\notag,
\end{gather}
and define the block matrix-valued function $K(\cdot,\mu) \in (W^{n}_p)^{rm\times rm}$ by
\begin{equation}\label{AAv-reduced}
K(\cdot,\mu):=\left(
\begin{array}{ccccc}
O_m & -I_m & O_m & \ldots & O_m \\
O_m & O_m & -I_m & \ldots & O_m \\
\vdots & \vdots & \vdots & \ddots & \vdots \\
O_m & O_m & O_m & \ldots & -I_m \\
A_0(\cdot,\mu) & A_1(\cdot,\mu) & A_2(\cdot,\mu) & \ldots & A_{r-1}(\cdot,\mu)\\
\end{array}\right),
\end{equation}
where $O_m$ and $I_m$ denote the null and identity matrices of dimension $m \times m$, respectively.
A vector-valued function $y(\cdot,\mu)\in(W^{n+r}_p)^{m}$ is a solution to the system \eqref{bound_z1} if and only if the vector-valued function \eqref{yx} is a solution to the system
\begin{gather*}\label{Cauchi-reduced-DE}
x'(t,\mu)+K(t,\mu)x(t,\mu)=g(t,\mu),
\quad t\in(a,b).
\end{gather*}

We consider an unknown matrix-valued function $Y(\cdot,\mu)$ of dimension $m \times rm$ with entries in $W^{n+r}_p$, and write it in the block form
\begin{equation*}\label{block-Y}
Y(\cdot,\mu) = \bigl(Y_0(\cdot,\mu), \ldots, Y_{r-1}(\cdot,\mu)\bigr),
\end{equation*}
where each block $Y_k(\cdot,\mu)$ ($k \in \{0, \ldots, r-1\}$) is an $m \times m$ matrix-valued function.
We denote by
\begin{equation*}\label{matrix_BYe}
\left[B(\mu)Y(\cdot,\mu)\right]:=\left(\left[B(\mu)Y_0(\cdot,\mu)\right],\ldots,\left[B(\mu)
Y_{r-1}(\cdot,\mu)\right]\right) \in \mathbb{C}^{rm\times rm}
\end{equation*}
the numerical block matrix of dimension $rm\times rm$. It consists of $r$ block columns $\left[B(\mu)Y_k(\cdot,\mu)\right]\in \mathbb{C}^{rm\times m}$ ($k \in \{0, \ldots, r-1\}$), where the $s$-th column of the block $\left[B(\mu)Y_k(\cdot,\mu)\right]$ is the result of the action of the operator $B(\mu)$ on the $s$-th column of the matrix-valued function $Y_k(\cdot,\mu)$ for each $s \in \{1,\ldots,m\}$.
Consider the following matrix boundary-value problem:
\begin{gather}
Y^{(r)}(t,\mu)+\sum\limits_{j=1}^{r}{A_{r-j}(t,\mu)Y^{(r-j)}(t,\mu)}=O_{m\times rm},
\quad t\in (a,b),\nonumber \\
\left[B(\mu)Y(\cdot,\mu)\right]=I_{rm}\label{eq2vk}.
\end{gather}
Here, $O_{m \times rm}$ is the zero matrix of dimension $m \times rm$, and $I_{rm}$ is the identity matrix of dimension $rm \times rm$.
Writing $Y(\cdot,\mu) = \left(y_{j,k}(\cdot,\mu)\right)_{\substack{j=1,\ldots,m\\ k=1,\ldots,rm}}$, problem \eqref{eq2vk} is a collection of $rm$ boundary-value problems of the form \eqref{bound_z1}, \eqref{bound_z2}, the right-hand sides of which do not depend on $\mu$. Therefore, the problem has a unique solution $Y(\cdot;\mu)$  for every $\mu\in \mathcal{B}(\mu_0, \varepsilon)$, due to condition $(\ast)$ of Definition \ref{defin_vp}. Moreover, by
$(\ast\ast)$, there holds the convergence
\begin{equation}\label{zb ym}
y_{j,k}(\cdot,\mu)\rightarrow y_{j,k}(\cdot,\mu_0) \quad\mbox{in} \quad W^{n+r}_p \quad\mbox{as}\quad \mu\to\mu_0
\end{equation}
for all $j \in \{1,\ldots,m\}$ and $k \in \{1,\ldots,rm\}$.

For every $k\in\{1, \ldots, rm\}$ and $\mu\in \mathcal{B}(\mu_0, \varepsilon)$, we define the
vector-valued function $x_k(\cdot,\mu)\in(W^{n+1}_p)^{rm}$ by formula \eqref{yx}, where we
take
$$
y(\cdot;\mu):=\big(y_{1,k}(\cdot;\mu),\ldots, y_{m,k}(\cdot;\mu)\big)^\top.
$$
Let $X(\cdot;\mu) \in (W^{n+1}_p)^{rm\times rm}$  denote the matrix-valued function such that its
$k$th column is $x_k(\cdot,\mu)$ for each $k\in\{1, \ldots, rm\}$. This function satisfies the matrix differential equation
\begin{equation}\label{X ym}
X'(t,\mu)+K(t,\mu)X(t,\mu)=O_{rm}, \quad t\in (a,b),
\end{equation}
where $O_{rm}$ is the null matrix of dimension $rm \times rm$.
We obtain $\det X(t;\mu)\neq0$ for all $t\in[a,b]$, since otherwise the columns of
$X(\cdot;\mu)$ and, hence, of $Y(\cdot;\mu)$ would be linearly dependent on $[a,b]$, contrary to~\eqref{eq2vk}.
Due to \eqref{zb ym}, we have the convergence $X(\cdot,\mu) \to X(\cdot,\mu_0)$ in $(W_p^{n+1})^{rm \times rm}$ as $\mu \to \mu_0$. Hence, $\bigl(X(\cdot,\mu)\bigr)^{-1} \to \bigl(X(\cdot,\mu_0)\bigr)^{-1}$ in $(W_p^{n+1})^{rm \times rm}$.
Since $X'(\cdot,\mu) \to X'(\cdot,\mu_0)$ in $(W_p^n)^{rm \times rm}$, and the pointwise multiplication map
\[
(W_p^n)^{rm \times rm} \times (W_p^{n+1})^{rm \times rm} \to (W_p^n)^{rm \times rm}
\]
is continuous (for $n \ge 1$ because $W_p^n$ is a Banach algebra, and for $n = 0$ due to the continuous multiplication $L_p \times W_p^1 \to L_p$ via the Sobolev embedding $W_p^1 \hookrightarrow L_\infty$), in view of \eqref{X ym} we conclude that
\[
K(\cdot,\mu) = -X'(\cdot,\mu)\bigl(X(\cdot,\mu)\bigr)^{-1} \to -X'(\cdot,\mu_0)\bigl(X(\cdot,\mu_0)\bigr)^{-1} = K(\cdot,\mu_0)
\]
in $(W_p^n)^{rm \times rm}$ as $\mu \to \mu_0$. Thus, by the definition of $K$ in \eqref{AAv-reduced},
the problem \eqref{bound_z1}, \eqref{bound_z2} satisfies Limit Condition~(I). Since convergent sequences are bounded, we have
\begin{equation}\label{ob Av}
\big\|A_{r-\ell}(\cdot,\mu)\big\|_{n,p}=O(1) \quad \mbox{as}\ \mu\to\mu_0,  \ \mbox{for all} \ \ell\in\{1, \ldots, r\}.
\end{equation}

\emph{Step 2.} We next prove that
\begin{equation}\label{ob Bv}
\|B(\mu)\|=O(1) \quad\mbox{as}\ \mu\to\mu_0.
\end{equation}

Suppose the contrary holds; then there exists a sequence $\left(\mu^{(k)}\right)_{k=1}^{\infty}\subset\mathcal{B}(\mu_0, \varepsilon)$ such that
\begin{equation}\label{ob Bev}
\mu^{(k)}\to\mu_0\quad\mbox{and}\quad 0<\left\|B\bigl(\mu^{(k)}\bigr)\right\|\to\infty \quad\mbox{as}
\ k\rightarrow \infty,
\end{equation}
where $\left\|B\bigl(\mu^{(k)}\bigr)\right\|\neq0$ for all $k\in \mathbb{N}$. For every
$k\in \mathbb{N}$, we choose a vector-valued function
$\omega_{k}\in(W^{n+r}_p)^{m}$ that satisfies the conditions
\begin{equation}\label{ob wv}
\|\omega_{k}\|_{n+r,p}=1 \quad \mbox{and} \quad \bigl\|B\bigl(\mu^{(k)}\bigr)\omega_{k}\bigr\|_{\mathbb{C}^{rm}}\geq \frac{1}{2}
\bigl\|B\bigl(\mu^{(k)}\bigr)\bigr\|.
\end{equation}
In addition, we put
\begin{gather*}
y\bigl(\cdot;\mu^{(k)}\bigr):=
\bigl\|B\bigl(\mu^{(k)}\bigr)\bigr\|^{-1}\omega_{k} \in (W^{n+r}_p)^{m},\\
f\bigl(\cdot;\mu^{(k)}\bigr):=
L\bigl(\mu^{(k)}\bigr)\,y\bigl(\cdot;\mu^{(k)}\bigr) \in (W^{n}_p)^{m},\\
c\bigl(\mu^{(k)}\bigr):=B\bigl(\mu^{(k)}\bigr)\,y\bigl(\cdot;\mu^{(k)}\bigr) \in \mathbb{C}^{rm}. \end{gather*}
Due to \eqref{ob Bev} and \eqref{ob wv}, there holds
\begin{equation}\label{zb ymal}
y\left(\cdot;\mu^{(k)}\right)\to0 \quad \mbox{in} \ (W^{n+r}_p)^{m}  \  \mbox{as} \ k\rightarrow \infty.
\end{equation}
Hence,
\begin{equation}\label{zb fmal}
f\left(\cdot;\mu^{(k)}\right)\to0 \quad \mbox{in}\ (W^{n}_p)^{m}  \  \mbox{as} \ k\rightarrow \infty
\end{equation}
because the problem \eqref{bound_z1}, \eqref{bound_z2} satisfies Limit Condition~(I), by Step 1.

Since the finite-dimensional space $\mathbb{C}^{rm}$ is locally compact, the bounds~\eqref{ob wv} yield
$$
1/2\leq\left\|c\left(\mu^{(k)}\right)\right\|_{\mathbb{C}^{rm}}\leq1.
$$
Indeed,
\begin{gather*}
\left\|c\left(\mu^{(k)}\right)\right\|_{\mathbb{C}^{rm}}\leq \left\|B\left(\mu^{(k)}\right)\right\| \left\|y\left(\cdot,\mu^{(k)}\right)\right\|_{n+r,p}\\
=\left\|B\left(\mu^{(k)}\right)\right\| \left\|B\left(\mu^{(k)}\right)\right\|^{-1} \left\|\omega_k\right\|_{n+r,p}=1
\end{gather*}
and thus
\begin{gather*}
\left\|c\left(\mu^{(k)}\right)\right\|_{\mathbb{C}^{rm}}= \left\|B\left(\mu^{(k)}\right)\left(
\left\|B\left(\mu^{(k)}\right)\right\|^{-1}\omega_k\right)\right\| = \left\|B\left(\mu^{(k)}\right)\right\|^{-1} \left\|B\left(\mu^{(k)}\right)\omega_k\right\|\geq \frac{1}{2}.
\end{gather*}
Hence, there exists a subsequence $$\left(c\left(\mu^{(k_s)}\right)\right)^\infty_{s=1}\subset \left(c\left(\mu^{(k)}\right)\right)^\infty_{k=1}$$ and a nonzero vector $c(\mu_0)\in \mathbb{C}^{rm}$ such that
\begin{equation}\label{zb cmal}
c\left(\mu^{(k_s)}\right)\to c(\mu_0) \quad \mbox{in} \ \mathbb{C}^{rm}  \  \mbox{as} \ s\rightarrow \infty.
\end{equation}
For every integer $s\in \mathbb{N}$, the vector-valued function $y\left(\cdot;\mu^{(k_s)}\right)\in\left(W^{n+r}_p\right)^{m}$ is the unique solution to the boundary-value problem
\begin{gather*}
L\left(\mu^{(k_s)}\right)y\left(t;\mu^{(k_s)}\right)=f\left(t;\mu^{(k_s)}\right),  \quad t\in (a,b),\\
B\left(\mu^{(k_s)}\right)y\left(\cdot;\mu^{(k_s)}\right)=c\left(\mu^{(k_s)}\right).
\end{gather*}
Due to \eqref{zb fmal} and \eqref{zb cmal} and the condition $(\ast\ast)$ of Definition \ref{defin_vp}, we conclude that the function $y\left(\cdot;\mu^{(k_s)}\right)$ converges to the unique solution $y(\cdot;\mu_0)$ of the boundary-value problem
\begin{gather}
L(\mu_0)y(t,\mu_0)=0,\quad t\in (a,b), \nonumber\\
B(\mu_0)y(\cdot;\mu_0)=c(\mu_0) \label{kr ym}
\end{gather}
in the space $(W^{n+r}_p)^{m}$ as $s\rightarrow \infty$. But $y(\cdot;\mu_0)\equiv0$ due to \eqref{zb ymal}.
We obtain a contradiction with the boundary condition~\eqref{kr ym}, in which $c(\mu_0)\neq0$.
This proves the claim  \eqref{ob Bv}.

\emph{Step 3.} Using the results of the previous steps, we now show that problem \eqref{bound_z1}, \eqref{bound_z2} satisfies Limit Condition (II). According to \eqref{ob Av} and \eqref{ob Bv}, there exist numbers $\gamma'>0$ and $\varepsilon'>0$ such that
\begin{equation}\label{ner 1}
\|(L(\mu),B(\mu))\|\leq\gamma' \quad \mbox{for every}  \  \mu\in\mathcal{B}(\mu_0, \varepsilon'),
\end{equation}
where $\|\cdot\|$ denotes the norm of the bounded operator \eqref{(L,B)vp}. We choose an arbitrary
vector-valued function $y\in(W^{n+r}_p)^{m}$ and set $f(\cdot;\mu):=L(\mu)y$ and $c(\mu):=B(\mu)y$. Hence,
\begin{equation}\label{riv 1}
y=(L(\mu),B(\mu))^{-1}(f(\cdot;\mu),c(\mu)) \quad \mbox{for every}  \  \mu\in\mathcal{B}(\mu_0, \varepsilon').
\end{equation}
Here, $(L(\mu),B(\mu))^{-1}$ denotes the inverse of the operator \eqref{(L,B)vp}; the invertibility
follows from condition $(\ast)$ of Definition \ref{defin_vp}.
Using \eqref{ner 1} and \eqref{riv 1}, we obtain,
for every $\mu\in\mathcal{B}(\mu_0, \varepsilon')$,
\begin{align*}
& \bigl\|B(\mu)y-B(\mu_0)y\bigr\|_{\mathbb{C}^{rm}}
\leq\bigl\|(f(\cdot;\mu),c(\mu))-
(f(\cdot;\mu_0),c(\mu_0))\bigr\|_{(W^{n}_p)^{m}\times\mathbb{C}^{rm}}
\\
=& \bigl\|(L(\mu),B(\mu)) y - (f(\cdot;\mu_0),c(\mu_0))\bigr\|_{(W^{n}_p)^{m}\times\mathbb{C}^{rm}}
\\
=& \bigl\|(L(\mu),B(\mu))\bigl(y - (L(\mu),B(\mu))^{-1}(f(\cdot;\mu_0),c(\mu_0))\bigr)\bigr\|_{(W^{n}_p)^{m}\times\mathbb{C}^{rm}}
\\
\leq & \gamma'\,\bigl\|y - (L(\mu),B(\mu))^{-1}(f(\cdot;\mu_0),c(\mu_0))\bigr\|_{n+r,p}
\\
= &\gamma'\,\bigl\|(L(\mu_0),B(\mu_0))^{-1}(f(\cdot;\mu_0),c(\mu_0))
-(L(\mu),B(\mu))^{-1}(f(\cdot;\mu_0),c(\mu_0))\bigr\|_{n+r,p},
\end{align*}
where we used \eqref{riv 1} at $\mu=\mu_0$ for the last identity. Here, the right-hand side tends to 0
as $\mu\to\mu_0$, by the condition $(\ast\ast)$ of Definition \ref{defin_vp}.
This and \eqref{ob Bv} show that  the boundary-value problem \eqref{bound_z1}, \eqref{bound_z2} satisfies Limit Condition~(II), which completes the proof of the necessity.

Let us prove \emph{the sufficiency} in Theorem \ref{nep v}. Suppose that the boundary-value problem \eqref{bound_z1}, \eqref{bound_z2} satisfies Condition (0) and Limit Conditions (I) and (II). We show that the solution of this problem depends  continuously in the space $(W^{n+r}_p)^{m}$ on the~parameter $\mu$
at $\mu_0$. We divide the proof into four steps.

\emph{Step 1.} For every $\mu\in \mathcal{M}$, we first consider the Cauchy problem
\begin{gather}\label{s1}
L(\mu)\hat{y}(t,\mu)=f(t,\mu),\quad t\in(a,b),\\
\hat{y}^{(j-1)}(a,\mu)=c_{j}(\mu),\quad j\in\{1,\ldots,r\}. \label{ku1}
\end{gather}
Here, for every $\mu$, the vector-valued function $f(\cdot,\mu)\in(W^{n}_p)^m$ and the vectors $c_{j}(\mu)
\in\mathbb{C}^{m}$ are arbitrary. The unique solution $\hat{y}(\cdot,\mu)$ of this problem belongs to the
space $(W^{n+r}_p)^{m}$.

We show that the convergence of the right-hand sides
\begin{equation}\label{r_s.1}
f(\cdot,\mu)\to f(\cdot,\mu_0)\quad\mbox{in} \ (W^{n}_p)^{m}  \ \mbox{as} \  \mu\to \mu_0,
\end{equation}
\begin{equation}\label{gsp}
\begin{gathered}
c_{j}(\mu)\rightarrow c_{j}(\mu_0)\quad \mbox{in}  \  \mathbb{C}^{m} \ \mbox{as} \  \mu\to \mu_0
 \ \mbox{for every} \  j\in\{1,\ldots,r\}
\end{gathered}
\end{equation}
implies the convergence of the solutions
\begin{equation}\label{gsx}
\hat{y}(\cdot,\mu)\to \hat{y}(\cdot,\mu_0)\quad\mbox{in} \ (W^{n+r}_p)^{m}  \ \mbox{as} \  \mu\rightarrow\mu_0.
\end{equation}
The Cauchy problem \eqref{s1}, \eqref{ku1} can be reduced to the following Cauchy problem for the
first-order system:
\begin{gather}\label{Cauchi-reduced-DE}
x'(t,\mu)+K(t,\mu)x(t,\mu)=g(t,\mu),
\quad t\in(a,b),\\
x(a,\mu)=c(\mu).\label{Cauchi-reduced-cond}
\end{gather}
Here, the matrix-valued function $K(\cdot,\mu)$ and the vector-valued function $g(\cdot,\mu)$ are the same as in Step 1 of the necessity proof. Moreover,
\begin{gather*}\label{ghggh}
x(\cdot,\mu):=\left(\hat{y}(\cdot,\mu),\hat{y}'(\cdot,\mu),\ldots,\hat{y}^{(r-1)}(\cdot,\mu)\right)^\top\in (W^{n+1}_p)^{rm}, \\
c(\mu):= \left(c_1(\mu),\ldots,c_r(\mu)\right)^\top\in \mathbb{C}^{rm}.
\end{gather*}
By assumption, the boundary-value problem~\eqref{bound_z1}, \eqref{bound_z2} satisfies Limit Condition
(I), hence,
\begin{equation*}\label{Kgsx}
K(\cdot,\mu)\to K(\cdot,\mu_0)\quad\mbox{in} \ (W^{n}_p)^{rm\times rm}  \ \mbox{as} \  \mu\rightarrow\mu_0.
\end{equation*}
Now,  conditions \eqref{r_s.1} and \eqref{gsp} imply the convergence of the right-hand sides of the problem \eqref{Cauchi-reduced-DE}, \eqref{Cauchi-reduced-cond} \begin{equation*}\label{r_g.1}
g(\cdot,\mu)\to g(\cdot,\mu_0)\quad\mbox{in} \ (W^{n}_p)^{rm}  \ \mbox{as} \  \mu\to\mu_0,
\end{equation*}
\begin{equation*}\label{ggp}
\begin{gathered}
c(\mu)\rightarrow c(\mu_0)\quad\mbox{in} \ \mathbb{C}^{rm} \ \mbox{as} \  \mu\to \mu_0.
\end{gathered}
\end{equation*}

Therefore, applying \cite[Theorem~1]{AtlasiukMikhailets20192} to the first-order Cauchy problem \eqref{Cauchi-reduced-DE}, \eqref{Cauchi-reduced-cond}, we obtain the convergence
\begin{equation*}
x(\cdot,\mu) \to x(\cdot,\mu_0) \quad \mbox{in} \ (W^{n+1}_p)^{rm} \quad \mbox{as} \ \mu \to \mu_0.
\end{equation*}
In particular, this yields the convergence of the last vector component
\begin{equation*}
\hat{y}^{(r-1)}(\cdot,\mu) \to \hat{y}^{(r-1)}(\cdot,\mu_0) \quad \mbox{in} \ (W^{n+1}_p)^{m} \quad \mbox{as} \ \mu \to \mu_0,
\end{equation*}
which, together with the convergence of the lower-order derivatives $\hat{y}^{(j)}(\cdot,\mu) \to \hat{y}^{(j)}(\cdot,\mu_0)$ in $(W^{n+1}_p)^{m}$ for all $j \in \{0, \ldots, r-2\}$, implies the required convergence
\begin{equation*}
\hat{y}(\cdot,\mu) \to \hat{y}(\cdot,\mu_0) \quad \mbox{in} \ (W^{n+r}_p)^{m} \quad \mbox{as} \ \mu \to \mu_0.
\end{equation*}
This establishes \eqref{gsx}.

\emph{Step 2.} We prove that the boundary-value problem~\eqref{bound_z1}, \eqref{bound_z2} satisfies the condition $(\ast)$ of  Definition \ref{defin_vp}, that is, for sufficiently small $d(\mu,\mu_0)$, the operator $(L(\mu),B(\mu))$ is invertible.

For every $\mu\in \mathcal{M}$ and $k\in\{0,\ldots,r-1\}$, we consider the matrix Cauchy problem
\begin{gather}\label{zad kosh1v}
Y_k^{(r)}(t,\mu)+\sum\limits_{j=1}^rA_{r-j}(t,\mu)Y_k^{(r-j)}(t,\mu)=O_{m},\quad t\in (a,b),
\end{gather}
with initial conditions
\begin{gather*}\label{zad kosh2v}
Y_k^{(j)}(t_0,\mu) = \delta_{k,j}I_m,\quad j \in \{0,\dots, r-1\}.
\end{gather*}
Here,
$$
Y_k(t,\mu)=\left(y_k^{\alpha,\beta}(t,\mu)\right)_{\alpha,\beta=1}^m
$$
is an unknown $m\times m$-matrix-valued function, the point $t_0 \in [a,b]$ is fixed, $\delta_{k,j}$ is
the Kronecker  symbol, $O_{m}$ and $I_m$ are as in \eqref{AAv-reduced}. Problem \eqref{zad kosh1v} consists of $m$ Cauchy problems of the form \eqref{s1},
\eqref{ku1} with $f=0$ for the unknown vector functions $\hat{y}(\cdot,\mu)$, which are the
columns of the matrix $Y_k(\cdot,\mu)$.

We write the solution of the homogeneous differential equation \eqref{bound_z1} in the form
\begin{equation}\label{solution_by_Yv}
y(\cdot,\mu)=\sum\limits_{k=0}^{r-1}Y_k(\cdot,\mu)q_k(\mu),
\end{equation}
where $q_k(\mu)\in\mathbb{C}^{m}$ are arbitrary column vectors, cf. \cite{Cartan1971}.
The right-hand sides of this problem do not depend on $\mu$, hence,
\begin{equation}\label{Zlimitv}
Y_k(\cdot,\mu)\to Y_k(\cdot,\mu_0) \quad\mbox{in} \ (W^{n+r}_p)^{m\times m} \ \mbox{as} \  \mu\to\mu_0,
\end{equation}
according to Step 1 and Theorem \ref{th1}. To justify the convergence of the block columns $[B(\mu)Y_k(\cdot,\mu)]$, we use the identity
\begin{equation}\label{B-decomp}
B(\mu)Y_k(\cdot,\mu) - B(\mu_0)Y_k(\cdot,\mu_0) = B(\mu)\bigl(Y_k(\cdot,\mu) - Y_k(\cdot,\mu_0)\bigr) + \bigl(B(\mu) - B(\mu_0)\bigr)Y_k(\cdot,\mu_0).
\end{equation}
By Limit Condition~(II), the linear operators $B(\mu) \colon (W_p^{n+r})^m \to \mathbb{C}^{rm}$ converge strongly to $B(\mu_0)$ as $\mu \to \mu_0$. Hence, by the uniform boundedness principle (the Banach--Steinhaus theorem), their operator norms are locally bounded: $\|B(\mu)\| \le C$ for all $\mu \in \mathcal{B}(\mu_0, \varepsilon_1)$. Therefore,
\[
\bigl\|B(\mu)\bigl(Y_k(\cdot,\mu) - Y_k(\cdot,\mu_0)\bigr)\bigr\| \le C \bigl\|Y_k(\cdot,\mu) - Y_k(\cdot,\mu_0)\bigr\|_{n+r,p} \to 0 \quad \mbox{as} \ \mu \to \mu_0,
\]
due to \eqref{Zlimitv}. Moreover, $\bigl(B(\mu) - B(\mu_0)\bigr)Y_k(\cdot,\mu_0) \to O_{rm \times m}$ columnwise as $\mu \to \mu_0$, directly by Limit Condition~(II). Thus, we obtain the convergence
\begin{equation}\label{BZlimit}
\begin{gathered}
\bigl([B(\mu)Y_0(\cdot,\mu)],\ldots,
[B(\mu)Y_{r-1}(\cdot,\mu)]\bigr)\to\\
\to\bigl([B(\mu_0)Y_0(\cdot,\mu_0)],\ldots,
[B(\mu_0)Y_{r-1}(\cdot,\mu_0)]\bigr)\quad \mbox{in}  \  \mathbb{C}^{rm\times rm} \
\mbox{as} \ \mu\to\mu_0.
\end{gathered}
\end{equation}
However, the limit matrix is nondegenerate, due to Limit
Condition (0) and \cite[Lemma 1]{AtlasiukMikhailets20192}. Therefore, one can find
$\varepsilon_1  >0 $ such that, for every  $\mu\in\mathcal{B}(\mu_0, \varepsilon_1)$,
\begin{equation}\label{BZdet}
\begin{gathered}
\det\bigl(M(L(\mu),B(\mu))\bigr)\neq0.
\end{gathered}
\end{equation}
Hence, by \cite[Lemma 1]{AtlasiukMikhailets20192}, the operator \eqref{(L,B)vp} is invertible.

\emph{Step 3.} We show that the boundary-value problem~\eqref{bound_z1}, \eqref{bound_z2} satisfies the condition $(\ast\ast)$ of Definition \ref{defin_vp}. Let us first analyze the case $f(\cdot,\mu)\equiv0$:
we consider a $\mu$-dependent semihomogeneous boundary-value problem
 \begin{equation}\label{vv}
    L(\mu)v(\cdot;\mu)\equiv 0,
\end{equation}
 \begin{equation}\label{vvk}
 B(\mu)v(\cdot;\mu) =c(\mu).
\end{equation}

According to Step 2, this problem has a unique solution  $v(\cdot;\mu)\in(W^{n+r}_p)^{m}$ for every $\mu\in \mathcal{B}(\mu_0, \varepsilon')$, where $\varepsilon' > 0$ is
sufficiently small. Assuming
 \begin{equation}\label{zb c2v}
  c(\mu)\rightarrow c(\mu_0) \quad \mbox{in}  \  \mathbb{C}^{rm}  \  \mbox{as}  \ \mu\rightarrow \mu_0,
 \end{equation}
we next show that
 \begin{equation}\label{zb v2v}
  v(\cdot,\mu)\rightarrow v(\cdot,\mu_0) \quad \mbox{in}  \  (W^{n+r}_p)^{m}  \  \mbox{as}  \ \mu\rightarrow \mu_0.
 \end{equation}

For every $\mu\in\mathcal{B}(\mu_0, \varepsilon')$, we write the general solution of
 \eqref{vv} in the form \eqref{solution_by_Yv}, that is
\begin{equation}\label{solution_by_vv}
v(\cdot,\mu)=\sum\limits_{k=0}^{r-1}Y_k(\cdot,\mu)q_k(\mu),
\end{equation}
for some  $q_k(\mu)\in\mathbb{C}^{m}$ and  matrix-valued functions $Y_k(\cdot,\mu)
\in(W^{n+r}_p)^{m\times m}$ as in Step 2. By virtue of \cite[Lemma 6]{AtlasiukMikhailets20191},
there holds
\begin{equation*}\label{rivn_matruz vv}
B(\mu)v(\cdot,\mu)=\sum _{k=0}^{r-1}B(\mu)(Y_k(\cdot,\mu)q_k(\mu)) = \sum _{k=0}^{r-1}\left[B(\mu)Y_k(\cdot,\mu)\right]q_k(\mu),
\end{equation*}
and we thus see that the boundary problem \eqref{vvk} is equivalent to
\begin{equation*}\label{rivn_matruz vv}
\sum _{k=0}^{r-1}\left[B(\mu)Y_k(\cdot,\mu)\right]q_k(\mu)=c(\mu).
\end{equation*}
This equation can be rewritten in the form of a linear algebraic system 
\begin{equation*}\label{rivn_matruz vv}
\big(\left[B(\mu)Y_0(\cdot,\mu)\right], \ldots, \left[B(\mu)Y_{r-1}(\cdot,\mu)\right]\big)q(\mu)=c(\mu),
 \ \mbox{or},  \
\left[B(\mu)Y(\cdot,\mu)\right]q(\mu)=c(\mu) ,
\end{equation*}
where $q(\mu):=(q_0(\mu),\ldots, q_{r-1}(\mu))^\top.$
Now,
\eqref{BZlimit}, \eqref{BZdet} and assumption \eqref{zb c2v} yield
\begin{equation*}\label{rivn_matruz vv}
q(\mu)=\left[B(\mu)Y(\cdot,\mu)\right]^{-1}c(\mu)\rightarrow\left[B(\mu_0)Y(\cdot,\mu_0)\right]^{-1}c(\mu_0)=q(\mu_0) \quad \mbox{as}  \ \mu\to \mu_0.
\end{equation*}
Therefore, \eqref{zb v2v} follows from  formulas \eqref{Zlimitv}, \eqref{solution_by_vv}, since
\begin{equation*}\label{rivn_matruz vv}
v(\cdot,\mu)=\sum _{k=0}^{r-1}Y_k(\cdot,\mu)q_k(\mu) \rightarrow \sum _{k=0}^{r-1}Y_k(\cdot,\mu_0)q_k(\mu_0)=v(\cdot,\mu_0)
\end{equation*}
in $(W^{n+r}_p)^{m}$ as $\mu\to \mu_0.$

\emph{Step 4.} We finally consider the general case of an inhomogeneous differential equation \eqref{bound_z1}. Suppose that conditions \eqref{zb c2v} and the condition
\begin{equation}\label{zb fvvy}
f(\cdot,\mu) \rightarrow f(\cdot,\mu_0) \quad \mbox{in}  \  (W^{n}_p)^{m}  \  \mbox{as}  \ \mu\to \mu_0
\end{equation}
are satisfied.
For every $\mu \in \mathcal{B}(\mu_0, \varepsilon_1)$, we set
$$z(\cdot;\mu)=y(\cdot;\mu)-\hat{y}(\cdot;\mu),$$
where the vector-valued function $y(\cdot;\mu)$ is a solution of the inhomogeneous boundary-value problem
\eqref{bound_z1}, \eqref{bound_z2} and $\hat{y}(\cdot;\mu)$ is a solution of the Cauchy problem \eqref{s1}, \eqref{ku1}, with $c_{j}(\mu)\equiv0$. Then, $z(\cdot;\mu)$ is a solution of the semihomogeneous boundary-value problem
\begin{gather*}\label{s2}
L(\mu)z(\cdot,\mu)\equiv0,\\
B(\mu)z(\cdot,\mu)=\tilde{c}(\mu),\\
\tilde{c}(\mu)=c(\mu)-B(\mu)\hat{y}(\cdot,\mu) \in \mathbb{C}^{rm}.
\end{gather*}
In Step 1 it was shown that $\hat{y}(\cdot,\mu)$ satisfies property \eqref{gsx} provided that condition
\eqref{zb fvvy} holds. Using the identity
\[
B(\mu)\hat{y}(\cdot,\mu) - B(\mu_0)\hat{y}(\cdot,\mu_0) = B(\mu)\bigl(\hat{y}(\cdot,\mu) - \hat{y}(\cdot,\mu_0)\bigr) + \bigl(B(\mu) - B(\mu_0)\bigr)\hat{y}(\cdot,\mu_0),
\]
the uniform boundedness of the norms $\|B(\mu)\|$, property \eqref{gsx}, and Limit Condition~(II), we obtain $B(\mu)\hat{y}(\cdot,\mu) \to B(\mu_0)\hat{y}(\cdot,\mu_0)$ in $\mathbb{C}^{rm}$ as $\mu \to \mu_0$. This fact, together with \eqref{zb c2v}, implies
 \begin{equation*}\label{zb c2tv}
  \tilde{c}(\mu)\rightarrow \tilde{c}(\mu_0) \quad \mbox{in}  \  \mathbb{C}^{rm}  \  \mbox{as}  \ \mu\rightarrow \mu_0.
 \end{equation*}
Hence, applying Step 3, we conclude that
\begin{equation*}\label{zb fvy}
z(\cdot,\mu) \rightarrow z(\cdot,\mu_0) \quad \mbox{in}  \  (W^{n+r}_p)^{m}  \  \mbox{as}  \ \mu\to \mu_0 ,
\end{equation*}
and we obtain the desired convergence from  formula \eqref{gsx}:
\begin{gather*}\label{zb fvy}
y(\cdot,\mu)= \hat{y}(\cdot,\mu)+z(\cdot,\mu)\rightarrow \hat{y}(\cdot,\mu_0)+z(\cdot,\mu_0)=y(\cdot,\mu_0)\\ \quad \mbox{in}  \  (W^{n+r}_p)^{m}  \  \mbox{as}  \ \mu\to \mu_0.
\end{gather*}
This proves the sufficiency.
\end{proof}

We supplement the previous result with an error estimate  for   the solution $y(\cdot;\mu)$,
namely, considering $y(\cdot;\mu)$ as an approximate solution of the problem \eqref{bound_z1} for the
parameter value $\mu_0$, \eqref{bound_z2}, we show that
$
\bigl\|y(\cdot;\mu_0)-y(\cdot;\mu)\bigr\|_{n+r,p}
$
is proportional to the discrepancy
\begin{equation*}\label{nevyuzka v}
\widetilde{d}_{n,p}(\mu):=
\bigl\|L(\mu)y(\cdot;\mu_0)-f(\cdot;\mu)\bigr\|_{n,p}+
\bigl\|B(\mu)y(\cdot;\mu_0)-c(\mu)\bigr\|_{\mathbb{C}^{rm}}.
\end{equation*}

\begin{theorem}\label{3.6.th-bound v}
Let $\mu_0 \in \mathcal{M}$ and assume that the boundary-value problem \eqref{bound_z1},
\eqref{bound_z2} satisfies Condition \textup{(0)} and Limit Conditions \textup{(I)} and \textup{(II)}.
Then there exist positive numbers $\varepsilon$, $\gamma_{1}$, and $\gamma_{2}$ such that
\begin{equation}\label{3.6.bound}
\begin{aligned}
\gamma_{1}\,\widetilde{d}_{n,p}(\mu)
\leq\bigl\|y(\cdot;\mu_0)-y(\cdot;\mu)\bigr\|_{n+r,p}\leq
\gamma_{2}\,\widetilde{d}_{n,p}(\mu),
\end{aligned}
\end{equation}
for all $\mu\in\mathcal{B}(\mu_0, \varepsilon)$. Here, $\varepsilon$, $\gamma_{1}$, and $\gamma_{2}$ do not depend on $y(\cdot;\mu_0)$ or  $y(\cdot;\mu)$.
\end{theorem}

\begin{proof}
The  left-hand side of~\eqref{3.6.bound} follows from \eqref{ner 1}, once $\varepsilon$ is chosen to be
smaller than or equal to $\varepsilon'$:
\begin{gather*}
\widetilde{d}_{n,p}(\mu)=\bigl\|L(\mu)y(\cdot;\mu_0)-L(\mu)y(\cdot;\mu)\bigr\|_{n,p}+
\bigl\|B(\mu)y(\cdot;\mu_0)-B(\mu)y(\cdot;\mu)\bigr\|_{\mathbb{C}^{rm}}
\\
\leq\bigl\|L(\mu)\bigr\|\bigl\|y(\cdot;\mu_0)-y(\cdot;\mu)\bigr\|_{n+r,p}+
\bigl\|B(\mu)\bigr\|\bigl\|y(\cdot;\mu_0)-y(\cdot;\mu)\bigr\|_{n+r,p}
\\ \leq \gamma'\bigl\|y(\cdot;\mu_0)-y(\cdot;\mu)\bigr\|_{n+r,p}.
\end{gather*}

To see the right-hand side, by Theorem~\ref{nep v}, the boundary-value problem \eqref{bound_z1}, \eqref{bound_z2} satisfies Definition~\ref{defin_vp} so that the operator \eqref{(L,B)vp} is invertible for every small $d(\mu,\mu_0)$. Moreover,
there holds the strong convergence
$$
(L(\mu),B(\mu))^{-1} \stackrel{s}{\longrightarrow} (L(\mu_0),B(\mu_0))^{-1}, \quad \mu\to\mu_0.
$$
Indeed, for arbitrary $f\in(W^{n}_p)^{m}$ and $c\in\mathbb{C}^{rm}$, the condition $(\ast\ast)$ of Definition~\ref{defin_vp}
yields
\begin{equation*}
(L(\mu),B(\mu))^{-1}(f,c)=:y(\cdot,\mu)\to
y(\cdot,\mu_0):=(L(\mu_0),B(\mu_0))^{-1}(f,c)
\end{equation*}
in the space $(W^{n+r}_p)^{m}$, as $\mu\to\mu_0$.
Hence, by the Banach--Steinhaus theorem, the norms of these inverse operators are bounded, namely, there exist positive numbers $\varepsilon$ and $\gamma_{2}$ such that
$$
\left\|(L(\mu),B(\mu))^{-1}\right\|\leq\gamma_{2} \quad\mbox{for every} \  \mu\in\mathcal{B}(\mu_0, \varepsilon).
$$
Thus, for every $\mu\in\mathcal{B}(\mu_0, \varepsilon)$, we conclude that
\begin{gather*}
\big\|y(\cdot,\mu_0)-y(\cdot,\mu)\big\|_{n+r,p}
=\|(L(\mu),B(\mu))^{-1}(L(\mu),B(\mu))
(y(\cdot,\mu_0)-y(\cdot,\mu))\|_{n+r,p}\\
\leq\gamma_{2}\,\|(L(\mu),B(\mu))
(y(\cdot,\mu_0)-y(\cdot,\mu))\|_{(W^{n}_p)^{m}\times\mathbb{C}^{rm}}=\gamma_{2}\,\widetilde{d}_{n,p}(\mu).
\end{gather*}
This yields the right-hand side of \eqref{3.6.bound}.
\end{proof}

The right-hand inequality in formula \eqref{3.6.bound} can be interpreted as the Ulam--Hyers  stability of solutions of the boundary-value problem with parameter as  $\mu \rightarrow \mu_0$ (see, for example, \cite{Rus,Hyers}).

\section{Limit theorems concerning the operators of the boundary-value problem}\label{Sec.3}

Our aim in this section is to present conditions  on the coefficients of differential expressions and
operators $B(\mu)$ under which $\big(L(\mu),B(\mu)\big)$ converges to the operator $\big(L(\mu_0),B(\mu_0)\big)$
in the strong and uniform operator topologies.
First, we formulate necessary and sufficient conditions for the strong and uniform convergence of the family
of operators $L(\mu)$ to the operator $L(\mu_0)$.

\begin{theorem}\label{eqival ym L}
Let $1\leq p\leq \infty$. The following convergence conditions are equivalent, when $\mu\to\mu_0$
in the metric space $\mathcal{M}$:
\begin{itemize}
\item [(I)] $A_{r-\ell}(\cdot,\mu) \rightarrow A_{r-\ell}(\cdot, \mu_0)$ in the Banach space
$(W^{n}_p)^{m\times m}$ for all $\ell\in\{1,\ldots, r\}$;
\item [(II)] $L(\mu) \rightarrow L(\mu_0)$ in the uniform operator topology;
\item [(III)] $L(\mu) \rightarrow L(\mu_0)$ in the strong operator topology.
\end{itemize}
\end{theorem}

Note that in the case where the metric parameter $\mu$ is a natural number, Theorem \ref{eqival ym L} was
proved in the paper \cite[Lemma 6.1]{NovAM2023}. The proof of the general case is similar, but we
present the details
for the convenience of the reader.

\begin{proof}[Proof of Theorem \ref{eqival ym L}]
It suffices to show that $(\mathrm{I})\Rightarrow (\mathrm{II})$ and $(\mathrm{III}) \Rightarrow (\mathrm{I}).$

If $(\mathrm{I})$ holds, we get for all $n > 0 $ and $y\in (W^{n+r}_p)^{m}$,
\begin{align*}
\bigl\|(L(\mu)-L(\mu_0))y\bigr\|_{n,p}&
\leq c_{n,p}\sum_{\ell=1}^r\bigl\|(A_{r-\ell}(\cdot,\mu)- A_{r-\ell}(\cdot, \mu_0))
\bigr\|_{n,p}\,\|y^{(r-\ell)}\|_{n,p}
\\
&\leq c_{n,p}\sum_{\ell=1}^r\bigl\|(A_{r-\ell}(\cdot,\mu)- A_{r-\ell}(\cdot, \mu_0))\bigr\|_{n,p}\,\|y\|_{n+r,p}
\end{align*}
because $W^{n}_p$ is a Banach algebra.
If $n=0$, then for every $y\in (W^{r}_p)^{m}$ we estimate
\begin{align*}
\bigl\|(L(\mu)-L(\mu_0))y\bigr\|_{0,p}&\leq
\sum_{\ell=1}^r\bigl\|(A_{r-\ell}(\cdot,\mu)- A_{r-\ell}(\cdot, \mu_0))\bigr\|_{0,p}\,
\|y^{(r-\ell)}\|_{\infty}\\
&\leq c_{r,p}'\sum_{\ell=1}^r\bigl\|(A_{r-\ell}(\cdot,\mu)- A_{r-\ell}(\cdot, \mu_0))
\bigr\|_{0,p}\,\|y\|_{r,p},
\end{align*}
where $c_{r,p}'$ is the norm of the bounded embedding operator $W^{r}_p\hookrightarrow C^{(r-1)}$.
These estimates imply $(\mathrm{II})$.

Suppose $(\mathrm{III})$ is valid.  This implies
\begin{equation}\label{vusok p}
\sum\limits_{\ell=1}^r A_{r-\ell}(\cdot,\mu)Y^{(r-\ell)}\to
\sum\limits_{\ell=1}^r A_{r-\ell}(\cdot, \mu_0)Y^{(r-\ell)},
\end{equation}
in $(W^{n}_p)^{m\times m}$  for all $Y \in(W^{n+r}_p)^{m\times m}$. Putting $Y(t):=I_m$ here yields $A_{0}(\cdot,\mu)\to A_{0}(\cdot, \mu_0)$. Moreover, choosing  $Y(t):=tI_m$ in \eqref{vusok p} gives us
\begin{equation*}
A_{1}(t,\mu)+A_{0}(t,\mu)t\to A_{1}(t, \mu_0)+A_{0}(t, \mu_0)t,
\end{equation*}
which implies that $A_{1}(\cdot,\mu)\to A_{1}(\cdot, \mu_0)$, and $Y(t):=t^{2}I_m$ yields
\begin{equation*}
2A_{2}(t,\mu)+2A_{1}(t,\mu)t+A_{0}(t,\mu)t^{2}\to 2A_{2}(t, \mu_0)+2A_{1}(t, \mu_0)t+A_{0}(t, \mu_0)t^{2}
\end{equation*}
and thus $A_{2}(\cdot,\mu)\to A_{2}(\cdot, \mu_0)$. Continuing this process  proves $(\mathrm{I})$.
\end{proof}

The following result will be  needed for treating the convergence of the operators $B(\mu)$. It plays a crucial role in proving the results of Section \ref{Sec.4}.

\begin{theorem}\label{opBinSobolev}
Let $1\leq p \leq\infty$ and $1/p + 1/p'=1$, and let  $t_0 \in [a, b]$, the matrices $\big(\alpha_{s}\big)_{s=0}^{n+r-1} \subset \mathbb{C}^{rm\times m}$ and the matrix-valued function $\Phi(\cdot)\in L_{p'}([a, b] ; \mathbb{C}^{rm\times m})$ be given.
\begin{itemize}
\item [(1)] The operator
\begin{equation} \label{diyaBnash}
By=\sum _{s=0}^{n+r-1} \alpha_{s}\,y^{(s)}(t_0)+\int_{a}^b \Phi(t)y^{(n+r)}(t){\rm d}t, \quad y(\cdot)\in (W_{p}^{n+r})^{m},
\end{equation}
acts continuously from $(W_{p}^{n+r})^{m}$ into $\mathbb{C}^{rm}$ and its norm satisfies
\begin{equation*}\label{normaB}
\|B\|\leq \gamma\max  \limits_s \{\|\alpha_{s}\|\} + \|\Phi\|_{L_{p'}},
\end{equation*}
where $\gamma > 0$ is a constant independent of $\alpha_{s}$ and $\Phi(\cdot)$.
\item [(2)] If $p \neq \infty$, then every bounded operator $B\colon (W^{n+r}_p)^m\rightarrow\mathbb{C}^{rm}$ admits a unique canonical representation of the form \eqref{diyaBnash}.
\end{itemize}
\end{theorem}

It should be noted that in the case of $p =\infty$, not all operators $B$ can be presented in the form
\eqref{diyaBnash}, since there are continuous operators $B$ that are defined by integrals over finitely additive
measures (see, for instance, \cite{Bhaskara,Dunford,KantAk1982}).

\begin{proof}[Proof of Theorem \ref{opBinSobolev}]
Let us prove claim $(1)$. From the continuity of the embedding $W^{n+r}_p\hookrightarrow C^{(n+r-1)}$ it
follows that there exists a constant $\gamma > 0$ such that
$
\|y\|_{C^{(n+r-1)}}\leq \gamma \|y\|_{n+r,p}.
$
Moreover,
\begin{align*}
\left\| \sum_{s=0}^{n+r-1}\alpha_s y^{(s)}(t_0) \right\|
&\leq \sum_{s=0}^{n+r-1}\|\alpha_s\| |y^{(s)}(t_0)|
\\
&\leq \max \limits_s \{\|\alpha_s\|\}\big\|y\big\|_{C^{(n+r-1)}}\leq \gamma \max \limits_s \{\|\alpha_s\|\}\big\|y\big\|_{n+r, p}.
\end{align*}
By H\"{o}lder's inequality
$$
\left\| \int_{a}^b \Phi(t)y^{(n+r)}(t){\rm d}t \right\| \leq \|\Phi\|_{L_{p'}} \left\|y^{(n+r)}\right\|_{L_{p}} \leq \|\Phi\|_{L_{p'}} \left\|y\right\|_{n+r, p}.
$$

Hence, the norm of the operator $B$ admits the estimate
$$
\|B\|\leq \gamma \max \limits_s\{\|\alpha_s\|\}+\|\Phi\|_{L_{p'}}.
$$

To show $(2)$, we recall the so-called Triple lemma (e.g. \cite[Chapter IV, Section 5]{KolmFom}):   if $E$, $E_1$,
$E_2$ are  Banach spaces and the continuous linear operators $B \colon E \rightarrow E_1$ and
$A \colon E\rightarrow E_2$  are  such that
$$
AE=E_2, \quad \ker B \supset \ker A,
$$
then there exists a linear operator $C \colon E_2 \rightarrow E_1$ such that $B=CA$.

Let $B\colon (W^{n+r}_{p})^m\rightarrow\mathbb{C}^{rm}$ be an arbitrary continuous linear operator.
Denote by $(W^{n+r}_{p,0})^{m}$ the space of vector-valued functions $f \in (W^{n+r}_{p})^{m}$ such that $f^{(s)}(t_0)=0$ for all  $s\in \{0, \ldots, n+r-1 \}$. Let us introduce the operator $A \colon (W^{n+r}_{p,0})^{m}\rightarrow (L_{p})^{m}$, $A f:=f^{(n+r)}$.
In the triple lemma, we choose $E:=(W^{n+r}_{p,0})^m$, $E_1:=\mathbb{C}^{rm}$, $E_2:=(L_{p})^{m}$,
and $\widetilde B:= B \mid (W^{n+r}_{p,0})^m$, which is the restriction of the operator $B$ to the space $(W^{n+r}_{p,0})^m$. According to the lemma, there exists a continuous linear operator
$C \colon (L_{p})^{m}\rightarrow\mathbb{C}^{rm}$ such that $\widetilde Bg=CAg$ for every
$g \in (W^{n+r}_{p,0})^m$. By the Riesz representation theorem,
there exists a unique matrix-valued function
$\Phi \in L_{p'}([a, b] ; \mathbb{C}^{rm\times m})$ such that $C h = \int_a^b \Phi(t) h(t){\rm d}t$ for all $h \in (L_p)^{m}$. We obtain
\begin{equation*}
\widetilde B g= C Ag=\int_{a}^b \Phi(t)g^{(n+r)}(t){\rm d}t \ \ \ \ \  \mbox{for all} \ g  \in (W^{n+r}_{p,0})^m.
\end{equation*}

Now, every  $f \in (W^{n+r}_p)^m$ has  a unique representation
$$
f(t)=\sum_{s=0}^{n+r-1}\frac{(t-t_0)^s}{s!}f^{(s)}(t_0)+g(t), \ \ \ \ \ \mbox{where} \ g \in (W^{n+r}_{p,0})^m.
$$

Denoting by $e_j$, $j=1,\ldots , m$, the canonical basis vectors of
$\mathbb{C}^{m}$, we define for each $s=0, \ldots , n+r-1$ the matrix $\alpha_s \in \mathbb{C}^{rm\times m}$ whose columns are given by
\begin{equation*}
\alpha_s e_j = B\left(\frac{(t-t_0)^s}{s!} e_j\right) \in \mathbb{C}^{rm}, \quad j=1,\ldots,m.
\end{equation*}
Hence, writing $f^{(s)}(t_0) = \sum_{j=1}^m x_{j,s} e_j$ in the basis $\{e_j\}_{j=1}^m$, we obtain
\begin{equation*}
B \left( \sum_{s=0}^{n+r-1}\frac{(t-t_0)^s}{s!}f^{(s)}(t_0) \right) =
\sum_{s=0}^{n+r-1}\alpha_s f^{(s)}(t_0)
\end{equation*}
and thus
$$
Bf=\sum_{s=0}^{n+r-1}\alpha_s f^{(s)}(t_0)+\widetilde B(g) =
\sum_{s=0}^{n+r-1}\alpha_s f^{(s)}(t_0)+\int_{a}^b \Phi(t)f^{(n+r)}(t){\rm d}t.
$$

To prove the uniqueness of the canonical representation, assume that $B$ also has a representation
\begin{equation*}
By=\sum _{s=0}^{n+r-1} \beta_{s}\,y^{(s)}(t_0)+\int_{a}^b \Psi(t)y^{(n+r)}(t){\rm d}t, \quad y \in (W_{p}^{n+r})^{m} .
\end{equation*}
We get
\begin{equation*}
\sum _{s=0}^{n+r-1}\left(\alpha_s- \beta_{s}\right)\,y^{(s)}(t_0) =  \int_{a}^b \left(\Psi(t)-\Phi(t)\right)y^{(n+r)}(t){\rm d}t \ \ \ \mbox{for all} \  y \in (W_{p}^{n+r})^{m}.
\end{equation*}
This relation is equivalent to the similar equality for square matrix functions
\begin{equation}
\label{JTnew1}
\sum _{s=0}^{n+r-1}\left(\alpha_s- \beta_{s}\right)\,Y^{(s)}(t_0)
= \int_{a}^b \left(\Psi(t)-\Phi(t)\right)Y^{(n+r)}(t){\rm d}t
\end{equation}
for all $ Y\in  (W_{p}^{n+r})^{m \times m}.$
The right-hand side of \eqref{JTnew1} is identically zero (matrix), if the elements of $Y$ are
polynomials of degree at most $ n+r-1$. For all $s , \ell \in \{0, \ldots, n+r-1 \}$ we choose such
matrices, also having the property
$$
Y_\ell^{(s)}(t_0)=\delta_{\ell s}I_m
$$
with the Kronecker delta $\delta_{\ell s}$.
This is possible by applying the Taylor formula for polynomials at the point $t_0$: every
element $y_{ij}$ of the matrix $Y=\left(y_{i,j}\right)_{i,j=1}^m$ can be chosen to be a polynomial of degree
$\leq n+r-1$ such that the vector of derivatives $(y_{i,j}^{(s)}(t_0))_{s=0}^{n+r-1}$ takes a predetermined value belonging to
$\mathbb{C}^{n+r}$.  This implies $ \alpha_s= \beta_{s}$ for all $s \in \{0, 1, \ldots, n+r-1 \}.$

We finally  prove that $\Psi(t)=\Phi(t)$ almost everywhere on $(a,b)$. Indeed, since the
left-hand side of \eqref{JTnew1} is null, we obtain
\begin{equation}\label{formyla6}
\int_{a}^b \left(\Psi(t)-\Phi(t)\right)Z(t){\rm d}t=O_{rm\times m}
\end{equation}
for all $Z \in (L_{p})^{m \times m}.$
In the scalar case, the claim follows by the standard description of the dual space of $L_p$. The matrix
case $m \geq 2$ can be reduced to the scalar case by using the $m^2$ numerical matrices
$E_{i,j} \in \mathbb{C}^{m \times m}$, $i,j=1, \ldots,m$, where the entry  in the position $(i,j)$ equals
one and the others are zero. Substituting the matrix functions $Z = \varphi E_{i,j}$ with suitable
$\varphi \in L_p$ into \eqref{formyla6}, we can show that all entries of $\Psi - \Phi$ are zero, which
proves the claim.
\end{proof}

We next formulate necessary and sufficient conditions for the strong and uniform convergence of the family  $B(\mu)$. To this end, we consider the following asymptotic conditions as $\mu\to\mu_0$.

\begin{enumerate}
\renewcommand{\labelenumi}{\alph{enumi})}
\renewcommand{\theenumi}{\alph{enumi})}
\item[(a)] $\alpha_s(\mu)\rightarrow\alpha_s(\mu_0)$ in $\mathbb{C}^{rm\times m}$ for every
$s\in\{0,\dots, n+r-1\}$;
\item[(b)]
$\left\|\Phi(\cdot,\mu)\right\|_{p'}=O(1)$;
\item[(c)] $\int\limits_a^t \Phi(\tau,\mu)d\tau\rightarrow\int\limits_a^t\Phi(\tau, \mu_0)d\tau$
in the space $\mathbb{C}^{rm\times m}$ for all $t\in(a,b]$;
\item[(d)] $\|\Phi(\cdot,\mu)-\Phi(\cdot, \mu_0)\|_{p'}\rightarrow0$.
\end{enumerate}

It is easy to see that condition $(d)$ is stronger than conditions $(b)$ and $(c)$.

\begin{theorem}\label{strongunifB}
Let $1\leq p<\infty$. The operators $B(\mu)$ converge strongly to the operator $B(\mu_0)$
as  $\mu\to\mu_0$ if and only if conditions $(a)$, $(b)$ and $(c)$ hold. The convergence is
uniform if and only if $(a)$ and $(d)$ are satisfied.
\end{theorem}

\begin{proof}
The sufficiency of conditions $(a)$ and $(d)$ for the uniform convergence of the operators $B(\mu)$
follows from the operator norm estimate for $B(\mu)$ in Theorem \ref{opBinSobolev}.

We proceed to prove that the conditions \textit{(a), (b)} and \textit{(c)} are sufficient for the strong convergence.
It is enough to show that the right-hand side of
\begin{gather}\label{formula2}
\left\|B(\mu)y -B(\mu_0)y\right\| \leq  \sum _{s=0}^{n+r-1}\left\|\alpha_s(\mu)- \alpha_s(\mu_0)\right\|\left|y^{(s)}(t_0)\right| +
\\ \nonumber
\left\| \int_{a}^b \big(\Phi(t,\mu)-\Phi(t,\mu_0)\big)y^{(n+r)}(t){\rm d}t\right\|
\end{gather}
converges to zero as $\mu\to\mu_0$ for each $y \in (W_p^{n+r})^m$. Condition $(a)$ implies that the first sum on the right-hand side tends to 0. For the integral term, condition $(c)$ ensures that $\int_a^b \big(\Phi(t,\mu) - \Phi(t,\mu_0)\big) \mathcal{X}_{[a,\tau]}(t)\,{\rm d}t \to O_{rm\times m}$ for all $\tau \in [a,b]$, and hence $\int_a^b \big(\Phi(t,\mu) - \Phi(t,\mu_0)\big) h(t)\,{\rm d}t \to 0$ for every step function $h$. Since the set of step functions is dense in $(L_p)^m$ for $1 \le p < \infty$, this together with the uniform boundedness $\|\Phi(\cdot,\mu)\|_{L_{p'}} = O(1)$ from condition $(b)$ implies, by the Banach--Steinhaus theorem, that the integral term in \eqref{formula2} converges to zero for every $y^{(n+r)} \in (L_p)^m$. This proves the sufficiency of conditions \textit{(a), (b)} and \textit{(c)}.

We next show that conditions \textit{(a), (b)} and \textit{(c)} are necessary for the strong convergence.
The strong convergence implies that
\begin{gather}\label{st anal Yk}
\sum _{s=0}^{n+r-1} \alpha_{s}(\mu) Y^{(s)}(t_0)+\int_{a}^b \Phi(t, \mu)Y^{(n+r)}(t){\rm d}t \rightarrow\\ %
\sum _{s=0}^{n+r-1} \alpha_{s}(\mu_0) Y^{(s)}(t_0)+\int_{a}^b \Phi(t, \mu_0)Y^{(n+r)}(t){\rm d}t \nonumber
\end{gather}
in $\mathbb{C}^{rm\times m}$ for all matrix-valued functions $Y  \in (W^{n+r}_p)^{m\times m}$.
Choosing $Y(t):= \frac{(t-t_0)^s}{s!} I_m$ with $s=0,1,\ldots, n+r-1$ in \eqref{st anal Yk}
yields $\alpha_{s}(\mu)\rightarrow  \alpha_{s}(\mu_0)$ as $\mu \rightarrow \mu_0$, proving that $(a)$ holds.

To show the necessity of $(b)$ and $(c)$, let $N = n+r$. For any $h \in L_p$, we consider the scalar function
$$
y(t) := \frac{1}{(N-1)!} \int_{t_0}^t (t-s)^{N-1} h(s)\,{\rm d}s \in W_p^N,
$$
which satisfies $y^{(s)}(t_0) = 0$ for all $s = 0, \ldots, N-1$ and $y^{(N)}(t) = h(t)$ almost everywhere on $[a,b]$. For each $j \in \{1,\ldots,m\}$, we test \eqref{st anal Yk} with the matrix-valued function $Y(t) := y(t) E_{j,j}$, where $E_{j,j} \in \mathbb{C}^{m\times m}$ is the standard matrix unit. Since $Y^{(s)}(t_0) = O_{m\times m}$ for all $s = 0, \dots, N-1$, the sum with $\alpha_s$ vanishes, and \eqref{st anal Yk} yields
\begin{equation}\label{zb Phi}
\int_{a}^b \Phi(t, \mu) h(t) e_j\,{\rm d}t \rightarrow \int_{a}^b \Phi(t, \mu_0) h(t) e_j\,{\rm d}t, \quad \mu \to \mu_0,
\end{equation}
for every $h \in L_p$. By the Banach--Steinhaus theorem, this weak$^*$ convergence in $(L_p)^* = L_{p'}$ implies that the $j$-th column of $\Phi(\cdot,\mu)$ is bounded in $(L_{p'})^{rm}$, which gives condition $(b)$. Choosing $h = \mathcal{X}_{[a,\tau]}$ for any $\tau \in (a,b]$ in \eqref{zb Phi} yields condition $(c)$.

Finally, we assume that $B(\mu)$ converges to $B(\mu_0)$ in operator norm, that is,
\begin{equation} \label{JT_new3}
\big\|B(\mu)-B(\mu_0)\big\|\rightarrow 0 \quad \text{as} \ \mu \to \mu_0.
\end{equation}
Since uniform convergence implies strong convergence, conditions $(a)$, $(b)$ and $(c)$ hold. To prove condition $(d)$, we introduce the bounded linear operator $J \colon (L_p)^m \to (W_p^N)^m$ defined by
\begin{equation*}
Jh(t) := \frac{1}{(N-1)!} \int_{t_0}^t (t-s)^{N-1} h(s)\,{\rm d}s, \quad h \in (L_p)^m.
\end{equation*}
Since $(Jh)^{(s)}(t_0) = 0$ for all $s = 0, \ldots, N-1$ and $(Jh)^{(N)} = h$, we have by representation \eqref{diyaBnash}:
\begin{equation*}
\bigl(B(\mu) - B(\mu_0)\bigr) Jh = \int_a^b \bigl(\Phi(t,\mu) - \Phi(t,\mu_0)\bigr) h(t)\,{\rm d}t.
\end{equation*}
Since $J$ is bounded, $\|Jh\|_{N,p} \le C_J \|h\|_{L_p}$, it follows that
\begin{align*}
\left\| \int_a^b \bigl(\Phi(t,\mu) - \Phi(t,\mu_0)\bigr) h(t)\,{\rm d}t \right\| &\le \|B(\mu) - B(\mu_0)\| \|Jh\|_{N,p}
\\
&\le C_J \|B(\mu) - B(\mu_0)\| \|h\|_{L_p}.
\end{align*}
Taking the supremum over all $h \in (L_p)^m$ with $\|h\|_{L_p} \le 1$, we conclude that
\begin{equation*}\label{rivn y}
\big\|\Phi(\cdot,\mu) - \Phi(\cdot,\mu_0)\big\|_{L_{p'}} \le C_J \|B(\mu) - B(\mu_0)\| \to 0 \quad \text{as} \ \mu \to \mu_0,
\end{equation*}
which proves condition $(d)$ and completes the proof.
\end{proof}


\section{Approximation by solutions of multipoint boundary-value problems} \label{Sec.4}

Let us next apply the above results to the approximation of solutions of an inhomogeneous
boundary-value
problem by solutions of a sequence of boundary-value problems with polynomial coefficients and multipoint
boundary conditions. We will consider the case $p < \infty $, since the case $p=\infty$ is
significantly different and will be omitted here.

To  formulate the statement of the problem, we consider a well-posed boundary-value problem
\begin{equation}\label{bound_pr_aaaaa}
(L_0y_0)(t):=y^{(r)}_0(t) + \sum\limits_{\ell=1}^rA_{r-\ell,0}(t)y^{(r-\ell)}_0(t)=f_0(t), \quad t\in(a,b),
\end{equation}
with inhomogeneous boundary conditions
\begin{equation}\label{bound_pr_bbbb}
B_0y_0:=\sum _{s=0}^{n+r-1} \alpha_{s,0}\,y^{(s)}_0(t_0)+\int_{a}^b \Phi_0(t)y^{(n+r)}_0(t){\rm d}t=c_0,
\end{equation}
when the matrix-valued functions $A_{r-\ell,0}(\cdot) \in (W_p^n)^{m\times m}$, the vector-valued function $f_0(\cdot) \in (W^n_p)^m$, the vector $c_0 \in \mathbb{C}^{rm}$, the numerical matrices $\alpha_{s,0} \in \mathbb{C}^{rm\times m}$ and the matrix-valued function $\Phi_0(\cdot)\in L_{p'}([a, b] ; \mathbb{C}^{rm\times m})$, $p \in [1, \infty)$, $p^{-1}+p'^{-1}=1$, are given. As we proved
in Theorem \ref{opBinSobolev}, every arbitrary inhomogeneous boundary condition for equation \eqref{bound_pr_aaaaa} admits a unique canonical representation of the form \eqref{bound_pr_bbbb}, where $t_0$ is an arbitrary fixed point of the interval $[a,b]$.

Consider simultaneously a sequence of multipoint boundary-value problems
\begin{equation}\label{bound_pr_aaaww}
(L_ky_k)(t):=y^{(r)}_k(t) + \sum\limits_{\ell=1}^rA_{r-\ell,k}(t)y^{(r-\ell)}_k(t)=f_k(t), \quad t\in(a,b),
\end{equation}
\begin{equation}\label{bound_pr_bbbww}
B_ky_k:=\sum _{s=0}^{n+r-1} \alpha_{s,k}\,y^{(s)}_k(t_0)+\sum\limits_{j=0}^{N(k)}{\beta_{j,k}
y^{(n+r-1)}_k(t_{j,k})}=c_k,  \ \ \ k=1,2,3, \ldots.
\end{equation}
Here, the elements of the matrix-valued functions $A_{r-\ell,k}$ belong to
some dense set  $\mathcal{F}$ in the space $(W_p^n)^{m\times m}$ and $G$ is a dense set in the space $(W_p^n)^{m}$ and $f_k \in G$ for all $k$. Moreover, for all indices, $\alpha_{s,k}, \beta_{j,k} \in
\mathbb{C}^{rm\times m}$ and the points $t_{j,k}$ belong to some dense
set $\mathcal{P}$ in the interval $[a,b]$, and we have
$$
f_k \rightarrow f_0, \quad c_k \rightarrow c_0 \ \ \ \mbox{as} \  k \rightarrow\infty.
$$

We next study the natural problem of the existence of a sequence of boundary-value problems \eqref{bound_pr_aaaww}, \eqref{bound_pr_bbbww}, whose solutions satisfy the asymptotic formula
$$
y_k \rightarrow y_0 \quad \mbox{in}\quad (W_p^{n+r})^{m}  \ \ \mbox{as} \ k\rightarrow\infty.
$$
We will give a positive answer to this question, which is based on  Theorem \ref{nep v}
on the continuity of the solutions of boundary-value problems with respect to a parameter
belonging to an abstract metric space $\mathcal{M}$.

In fact, we set $\mathcal{M}=\mathbb{Z}_+$ and introduce a metric on $\mathcal{M}$ by
\[
d(n,m) = d(m,n) =
\begin{cases}
\left| \dfrac{1}{n} - \dfrac{1}{m} \right|, & n \neq 0,\, m \neq 0, \\[6pt]
\dfrac{1}{n}, & m = 0, n \in \mathbb{N},\\[6pt]
0, & n = m = 0.
\end{cases}
\]
Then, 0 is the only limit point in the metric space $\big(\mathbb{Z}_+, d\big)$ and
there holds
$$
d(0,n)\rightarrow 0 \Leftrightarrow n\rightarrow\infty.
$$

We present the main result of this section.

\begin{theorem}\label{thapro}
Assume that the homogeneous boundary-value problem has only the trivial solution. Then,
there exists a sequence of well-posed boundary-value problems of the form \eqref{bound_pr_aaaww}, \eqref{bound_pr_bbbww} with polynomial coefficients such that
\begin{itemize}
\item[(i)] if $1<p<\infty$
\begin{gather*}
\big\| \left(L_k, B_k\right)^{-1}-\left(L_0, B_0\right)^{-1}\big\|\rightarrow 0, \quad \text{as} \quad k\rightarrow \infty;
\end{gather*}
\item[(ii)] if $p=1$
\begin{gather*}
\left(L_k, B_k\right)^{-1}  \stackrel{s}{\longrightarrow} \left(L_0, B_0\right)^{-1} \  \text{and} \  y_k \rightarrow y_0 \ \text{in} \  (W_1^{n+r})^m, \quad\text{as} \ k\rightarrow \infty.
\end{gather*}
\end{itemize}
\end{theorem}

\begin{proof}[Proof of Theorem \ref{thapro}] We begin by proving the statement $(i)$. Since, by the assumption of the theorem, the operator $(L_0,B_0)$ is invertible, it suffices to show that there exists a sequence of differential operators $\{L_k\}$ with polynomial coefficients and a sequence of operators $\{B_k\}$ of the form \eqref{bound_pr_bbbww} with the property that
\begin{gather*}
\big\| \left(L_k, B_k\right)-\left(L_0, B_0\right)\big\|\rightarrow 0, \quad \text{as} \quad k\rightarrow \infty.
\end{gather*}
The density of algebraic polynomials in Sobolev spaces with $1\leq p<\infty$ and Theorem
\ref{eqival ym L}  imply that one can choose the coefficients  $A_{r-\ell,k}$ in
\eqref{bound_pr_aaaww} in such a way that their elements are algebraic polynomials and $\big\| L_k-L_0\big\|\rightarrow 0 $. Let us now prove the existence of a suitable sequence of operators $\left \{B_k\right \}_{k=1}^\infty$. According to Theorem \ref{opBinSobolev}, the operator $B_0$ admits a unique canonical representation
\begin{equation*}
B_0y=\sum _{s=0}^{n+r-1} \alpha_{s,0}\,y^{(s)}(t_0)+\int_{a}^b \Phi_0(t)y^{(n+r)}(t){\rm d}t, \quad y(\cdot)\in (W_{p}^{n+r})^{m},
\end{equation*}
where  $\Phi_0 \in \left(L_{p'}\right)^{rm\times m}$. The density of continuous matrix functions
in the space $\left(L_{p'}\right)^{rm\times m}$ and the uniform continuity of continuous functions
in a compact metric space imply that the set $\mathcal{F}$ of matrix-valued functions with step elements is dense in the space $\left(L_{p'}\right)^{rm\times m}$. Moreover, due to the density of the set of points $\mathcal{P}$ in $[a,b]$, we may assume without loss of generality
that the ends of all steps in the step elements belonging to $\mathcal{F}$ are contained in
the set $\mathcal{P}$.

Therefore, due to Theorem \ref{strongunifB}, there exists a sequence of operators $\{B_k\}$
of the form
\begin{equation}\label{formula5}
B_ky=\sum _{s=0}^{n+r-1} \alpha_{s,k}\,y^{(s)}(t_0)+\int_{a}^b \Phi_k(t)y^{(n+r)}(t){\rm d}t,
\end{equation}
where $\Phi_k \in \mathcal{F}$. It remains to show that operators of the form \eqref{formula5} will correspond to multipoint boundary conditions and can be rewritten in the form \eqref{bound_pr_bbbww}.

The step matrix-valued function $\Phi_k \in (L_{p'})([a,b];\mathbb{C}^{rm\times m})$ can be written as
\begin{equation}\label{formula6}
\Phi_k(t) = \sum\limits_{j=0}^{N(k)-1}C_{j,k}\mathcal{X}_{[t_{j,k},t_{j+1,k})}(t),
\end{equation}
where $C_{j,k} \in \mathbb{C}^{rm\times m}$, $a=t_{0,k}<t_{1,k}<\ldots<t_{N(k),k}=b$ is a partition of the interval $[a,b]$, and $\mathcal{X}_{[c,d)}$ denotes the characteristic function of an interval $[c,d)$ (with the last interval being closed, $[t_{N(k)-1,k}, t_{N(k),k}]$). Integrating by parts yields
\begin{align*}
\int_{a}^b \Phi_k(t)y^{(n+r)}(t){\rm d}t
&= \sum\limits_{j=0}^{N(k)-1} C_{j,k}\int_{t_{j,k}}^{t_{j+1,k}} y^{(n+r)}(t){\rm d}t
\\
&= \sum\limits_{j=0}^{N(k)-1} C_{j,k}\bigl(y^{(n+r-1)}(t_{j+1,k})-y^{(n+r-1)}(t_{j,k})\bigr)
\\
&= C_{N(k)-1,k}y^{(n+r-1)}(b) - C_{0,k}y^{(n+r-1)}(a)
\\
&\quad + \sum\limits_{j=1}^{N(k)-1} (C_{j-1,k}-C_{j,k})y^{(n+r-1)}(t_{j,k})
\\
&= \sum\limits_{j=0}^{N(k)} \beta_{j,k}y^{(n+r-1)}(t_{j,k}),
\end{align*}
where
\begin{gather*}
\beta_{0,k} := -C_{0,k}, \qquad \beta_{N(k),k} := C_{N(k)-1,k},
\\
\beta_{j,k} := C_{j-1,k}-C_{j,k} = -\triangle \Phi_k(t_{j,k}) \in \mathbb{C}^{rm\times m} \quad (1\leq j \leq N(k)-1),
\end{gather*}
and $\triangle \Phi_k(t) := \Phi_k(t+) -\Phi_k(t-)$ is the jump of the function $\Phi_k$ at the point $t$.
We obtain that $B_k$ is indeed of the form \eqref{bound_pr_bbbww}. This completes the proof of statement $(i)$.

To prove the second assertion of the theorem, we will need  the following
result, which may be known for experts, but we nevertheless present the proof.

\begin{theorem}\label{thseq}
Given $f \in L_\infty(a,b)$, there is a sequence $\{ f_k\}_{k=1}^\infty$ of
step functions on the interval $[a,b]$ which converges to $f$ in the weak* topology of
$L_\infty(a,b)$. The sequence of step functions can be chosen so that the ends of all steps belong to a dense subset $\mathcal{P} \supset \{a,b\}$  in $[a,b]$.
\end{theorem}

\begin{proof}[Proof of Theorem \ref{thseq}]
By an affine transform, we may consider the case $(a,b)=(-\pi, \pi)$ without loss of generality (assuming that $\{-\pi, \pi\} \subset \mathcal{P}$). It is well known (see e.g. \cite[Chapter II, Section 2]{hoffman}) that the sequence of Ces\`{a}ro means $\{\sigma_{k}(f)\}$ of partial sums of its Fourier series converges to $f$ in the weak* topology of the space $L_\infty(-\pi, \pi)=(L_1(-\pi, \pi))^*$.

For each $k \in \mathbb{N}$, the Ces\`{a}ro mean $\sigma_{k}(f)$ is a continuous function on $[-\pi, \pi]$. Since every continuous function on $[-\pi, \pi]$ is a uniform limit of step functions whose step endpoints belong to $\mathcal{P}$, for each $k \in \mathbb{N}$ there exists a step function $f_k$ with step endpoints in $\mathcal{P}$ such that
\begin{equation*}
\|f_k - \sigma_{k}(f)\|_\infty \le \frac{1}{k}.
\end{equation*}
Then, for every $g \in L_1(-\pi, \pi)$, we obtain
\begin{align*}
\left|\int_{-\pi}^\pi (f_k(t)-f(t))g(t)\,{\rm d}t \right|
&\le \left|\int_{-\pi}^\pi (f_k(t)-\sigma_k(f)(t))g(t)\,{\rm d}t \right|
\\
&\quad + \left|\int_{-\pi}^\pi (\sigma_k(f)(t)-f(t))g(t)\,{\rm d}t \right|
\\
&\le \|f_k - \sigma_k(f)\|_\infty \|g\|_{L_1} + \left|\int_{-\pi}^\pi (\sigma_k(f)(t)-f(t))g(t)\,{\rm d}t \right|
\\
&\le \frac{\|g\|_{L_1}}{k} + \left|\int_{-\pi}^\pi (\sigma_k(f)(t)-f(t))g(t)\,{\rm d}t \right| \to 0 \quad \text{as } k\to\infty.
\end{align*}
This proves that the sequence of step functions $\{f_k\}_{k=1}^\infty$ converges to $f$ in the weak* topology of $L_\infty(-\pi, \pi)$.
\end{proof}

Let us complete the proof of the assertion $(ii)$ of Theorem \ref{thapro}. Consider first the scalar case $m=1$. By Theorem \ref{thseq}, each entry of the matrix-valued function $\Phi_0 \in (L_\infty)^{r\times 1}$ can be approximated in the weak* topology of $L_\infty(a,b)$ by a sequence of step functions $\Phi_k$ whose step ends belong to $\mathcal{P}$. In particular, for every $\tau \in [a,b]$, testing with the characteristic function $\mathcal{X}_{[a,\tau]} \in L_1(a,b)$ yields
\begin{equation*}
\int_a^b \bigl(\Phi_k(t) - \Phi_0(t)\bigr)\mathcal{X}_{[a,\tau]}(t)\,{\rm d}t = \int_a^\tau \bigl(\Phi_k(t) - \Phi_0(t)\bigr)\,{\rm d}t \to O_{r\times 1} \quad \text{as } k\to\infty,
\end{equation*}
which is condition $(c)$ of Theorem \ref{strongunifB}. Moreover, by the Banach--Steinhaus theorem, the weak* convergent sequence $\{\Phi_k\}$ is uniformly bounded in $(L_\infty)^{r\times 1}$, which guarantees condition $(b)$ of Theorem \ref{strongunifB}. Setting $\alpha_{s,k} := \alpha_{s,0}$ (condition $(a)$), Theorem \ref{strongunifB} implies that the boundary operators satisfy $B_k \stackrel{s}{\longrightarrow} B_0$ as $k\to\infty$.

Since the differential operators satisfy $\|L_k - L_0\| \to 0$ (and thus $L_k \stackrel{s}{\longrightarrow} L_0$), Limit Condition (I) holds, while $B_k \stackrel{s}{\longrightarrow} B_0$ is Limit Condition (II). Since the limit problem $(L_0, B_0)$ has only the trivial solution by assumption (Condition (0)), Theorem \ref{nep v} implies that the solutions depend continuously on the parameter, which is equivalent to $(L_k, B_k)^{-1} \stackrel{s}{\longrightarrow} (L_0, B_0)^{-1}$ and $y_k \to y_0$ in $(W_1^{n+r})^m$ as $k\to\infty$. (We emphasize that strong convergence of operators alone does not in general guarantee strong convergence of their inverses; this conclusion relies on Theorem \ref{nep v}.)

The case $m\geq2$ reduces to the scalar one in the same way as in the proof of assertion $(i)$.
\end{proof}

In the case of $p=1$, there arises a natural question: what conditions should be imposed on the operator $\left(L_0, B_0\right)$ in order to assure that there exists an approximating sequence of boundary-value problems $(L_k, B_k)$ of the form \eqref{bound_pr_aaaww}, \eqref{bound_pr_bbbww} whose inverse operators converge in the uniform operator topology? Note that uniform convergence of inverse operators always holds when $1<p<\infty$ (Theorem \ref{thapro}\,(i)). In the case $p=1$, the answer is given by the following theorem, provided that the multipoint nodes can be chosen arbitrarily in $[a,b]$.

\begin{theorem}\label{thapro2}
Let $p=1$ and assume that the homogeneous boundary-value problem $(L_0, B_0)$ has only a trivial solution. If the multipoint nodes can be chosen arbitrarily in $[a,b]$, then there exists a sequence of well-posed boundary-value problems $(L_k, B_k)$ of the form \eqref{bound_pr_aaaww}, \eqref{bound_pr_bbbww} with polynomial coefficients such that
\begin{equation}\label{sulzb}
\big\| \left(L_k, B_k\right)^{-1}-\left(L_0, B_0\right)^{-1}\big\|\rightarrow 0 \quad \text{as} \quad k\rightarrow \infty,
\end{equation}
if and only if each entry of the matrix-valued function $\Phi_0$ in the representation \eqref{diyaBnash} of $B_0$ is equal to a regulated function almost everywhere on $[a,b]$.
\end{theorem}

Recall that a function on an interval $[a,b]$ is called regulated if it has finite left and right limits at every point of $[a,b]$ (see, for example, \cite{Bourbaki}).

\begin{proof}[Proof of Theorem \ref{thapro2}]
\emph{Necessity.} Assume that there exists a sequence of well-posed boundary-value problems $(L_k, B_k)$ of the form \eqref{bound_pr_aaaww}, \eqref{bound_pr_bbbww} such that \eqref{sulzb} holds. Since the operator $(L_0, B_0)$ is invertible, the convergence of the inverse operators in norm is equivalent to
\begin{gather*}
\big\| L_k-L_0\big\|\rightarrow 0 \quad \text{and} \quad \big\| B_k-B_0\big\|\rightarrow 0 \quad \text{as} \quad k\rightarrow \infty.
\end{gather*}
By Theorem \ref{strongunifB}, the uniform convergence $\|B_k - B_0\| \to 0$ implies condition $(d)$, which means that
\begin{equation*}
\|\Phi_k - \Phi_0\|_{L_\infty([a,b])} \to 0 \quad \text{as} \quad k \to \infty,
\end{equation*}
where each $\Phi_k$ is a step matrix-valued function corresponding to the multipoint boundary condition $B_k$ (see \eqref{formula6}). Hence, each entry of $\Phi_0$ is the limit in the $L_\infty$-norm of a sequence of step functions. By the classical Bourbaki theorem \cite{Bourbaki}, the closure of the set of step functions in the uniform metric on a compact interval is precisely the space of regulated functions. Therefore, each entry of $\Phi_0$ coincides almost everywhere with a regulated function.

\emph{Sufficiency.} Conversely, suppose that each entry of $\Phi_0$ is equal almost everywhere to a regulated function on $[a,b]$. By \cite{Bourbaki}, every regulated function is the uniform limit of a sequence of step functions. Hence, there exists a sequence of step matrix-valued functions $\Phi_k \in \mathcal{F}$ on $[a,b]$ such that $\|\Phi_k - \Phi_0\|_{L_\infty([a,b])} \to 0$ as $k \to \infty$. By letting the nodes of the partition for $\Phi_k$ serve as the multipoint nodes in $[a,b]$ (which are allowed to be chosen arbitrarily), we construct the multipoint boundary operators $B_k$ as in \eqref{formula5}, \eqref{formula6} with $\alpha_{s,k} := \alpha_{s,0}$. Then conditions $(a)$ and $(d)$ of Theorem \ref{strongunifB} are satisfied, so $\|B_k - B_0\| \to 0$. Approximating the coefficients of $L_0$ by algebraic polynomials yields $\|L_k - L_0\| \to 0$. Since $(L_0, B_0)$ is invertible, it follows that the operators $(L_k, B_k)$ are invertible for sufficiently large $k$ and condition \eqref{sulzb} holds.
\end{proof}

Our assumption in Theorem \ref{thapro2} that the multipoint nodes are chosen arbitrarily in $[a, b]$ is essential.

\begin{example}
\upshape
Let us consider the Sturm--Liouville differential equation on the interval $(a,b)=(0,1)$ (where $a=0$ and $b=1$) with integral boundary conditions of integer and/or fractional orders
\begin{equation}\label{ex1}
Ly=y'' - q(t)y=f(t), \quad t\in(0,1),
\end{equation}
\begin{equation}\label{ex2}
By=\sum _{j=1}^{N} \int_{0}^{1}\beta_{j}(t)\big(\mathrm{D}_{0+}^{(l_j)}y\big)(t)
{\rm d}t=c  \in \mathbb{C}^2,
\end{equation}
where the functions $\{q, f\}\subset W_p^1(0,1)$, column vector-valued functions $\beta_{j}(\cdot)\in \big(L^\infty(0,1)\big)^{2 \times 1}$, and $\big(\mathrm{D}_{0+}^{(l_j)}y\big)$ denote Caputo derivatives of fractional/integer orders (see, for example, \cite{Caputo})
\begin{equation*}
0 \leq l_1<l_2<\ldots < l_N< 3-1/p, \quad 1\leq p <\infty.
\end{equation*}
If the characteristic matrix $M(L,B) \in \mathbb{C}^{2 \times 2}$ of an inhomogeneous boundary-value problem \eqref{ex1}, \eqref{ex2} is non-degenerate, then according to \cite[Corollary 2.4]{NovAM2023} this boundary-value problem has a unique solution $y(\cdot) \in W^3_p(0,1)$. Based on Theorem \ref{thapro}, there exists a sequence of boundary-value problems of the form \eqref{ex1} with polynomial functions $q_k$ and $f_k$ (the polynomials $f_k$ approximating $f$ in $W_p^1(0,1)$ are chosen as part of the construction), and multipoint boundary conditions of a special form
\begin{equation*} B_ky=\alpha_{0,k}y(0)+\alpha_{1,k}y'(0)+\sum_{j=0}^{N(k)}\beta_{j,k}y''(t_{j,k})=c_k,  \end{equation*}
where $\beta_{j,k} \in \mathbb{C}^2$, and $\{t_{j,k}\}$ are some rational points from $[0,1]$, such that
\begin{equation*}
\big\|y-y_k\big\|_{3,p} \rightarrow 0, \quad k \rightarrow \infty,
\end{equation*}
provided that $\big\|q-q_k\big\|_{1,p} \rightarrow 0$, $\big\|f-f_k\big\|_{1,p} \rightarrow 0$, and $c_k \rightarrow c$.

Moreover, if $p>1$ then
\begin{equation*}
\big\|(L,B)^{-1}-(L_k,B_k)^{-1}\big\| \rightarrow 0, \quad k \rightarrow \infty.
\end{equation*}
\end{example}

\section{Acknowledgment}

The work of the first named author was funded by Postdoctoral Fellowship EU-MSCA4Ukraine (number: 1244691,
WBS-number: 4100609). This project has received funding through the MSCA4Ukraine project, which is funded by
the European Union. Views and opinions expressed are however those of the author only and do not necessarily
reflect those of the European Union, the European Research Executive Agency or the MSCA4Ukraine Consortium.

Neither the European Union nor the European Research  Executive Agency, nor the MSCA4 Ukraine Consortium as
whole nor any individual member institution of the MSCA4Ukraine Consortium can be held responsible for them.

The second named author would like to thank the Isaac Newton Institute for Mathematical Sciences, Cambridge, for support and hospitality during the ``Solidarity Program'' where work on this paper was undertaken. This work was supported by ``EPSRC grant no EP/R014604/1''. The author wishes to thank the Department of Mathematics, King's College London, for their hospitality.

The authors are grateful to the reviewers for their valuable remarks and suggestions, which helped to improve the article.

\end{document}